\documentclass[oneside,english,11pt]{amsart}
\usepackage[T1]{fontenc}
\usepackage{textcomp}
\usepackage{url}
\usepackage[utf8]{inputenc}
\usepackage{amstext}
\usepackage{amsthm}
\usepackage{amssymb}
\usepackage{xcolor}
\usepackage{enumitem}
\usepackage{comment}
\usepackage[a4paper]{geometry}
\usepackage{orcidlink}
\makeatletter
\numberwithin{equation}{section}
\numberwithin{figure}{section}
\theoremstyle{plain}
\newtheorem{thm}{\protect\theoremname}
\newtheorem{rem}[thm]{Remark}
\newtheorem{prop}[thm]{\protect\propositionname}
\newtheorem{lem}[thm]{\protect\lemmaname}
\newtheorem{cor}[thm]{\protect\corollaryname}
\numberwithin{thm}{section}
\makeatother

\usepackage{babel}
\providecommand{\lemmaname}{Lemma}
\providecommand{\propositionname}{Proposition}
\providecommand{\theoremname}{Theorem}
\providecommand{\corollaryname}{Corollary}
\allowdisplaybreaks

\begin{document}
\title[Discrepancy estimates for non-smooth convex bodies]{Discrepancy estimates for multi-dimensional non-smooth convex bodies: a case study}
\author[]{Roberto Bramati\orcidlink{0000-0002-5672-2139}}
\author[]{Luca Brandolini\orcidlink{0000-0002-9670-9051}}
\author[]{Alessandro Monguzzi\orcidlink{0000-0003-3233-5000}}
\address{Dipartimento di Ingegneria Gestionale, dell'Informazione e
della Produzione, Universit{\`a} degli Studi di Bergamo, Viale G. Marconi 5, 24044,
Dalmine BG, Italy}
\email{roberto.bramati@unibg.it}
\email{luca.brandolini@unibg.it}
\email{alessandro.monguzzi@unibg.it}
\thanks{All the authors are member of Indam-GNAMPA}
\keywords{discrepancy theory, Fourier analysis, non-smooth convex bodies, irregularities of distribution}

\subjclass[2020]{11K38, 42B10, 42B20}

\begin{abstract}
We study $L^2$-averaged discrepancies of finite sequences of points in the torus $\mathbb{T}^d$ with respect to translated and dilated copies of convex bodies with non-smooth boundary. Under suitable anisotropic assumptions on the decay of the Fourier transform of the body, we prove matching lower and upper bounds for the averaged discrepancy, obtaining the rate $
N^{1 - \frac{d+1}{d^2+d-1}}$.
This yields an intermediate regime between smooth convex bodies and polytopes and recovers the known exponent $2/5$ in dimension $d=2$. The argument relies on harmonic analysis techniques combined with averaging procedures adapted to the anisotropic setting. As an application, we analyze a class of convex bodies exhibiting mixed geometric features, including flat regions, curved parts, and edges.
\end{abstract}
\maketitle

\section{Introduction}
Let $\Omega\subseteq \mathbb{T}^d$ be a measurable set and let
$\mathcal P_N=(z_1,\dots,z_N)$ be a finite sequence of points in $\mathbb{T}^d$.
A basic question in geometric discrepancy theory is to quantify how well
$\mathcal P_N$ approximates the Lebesgue measure of translated and dilated
copies of $\Omega$. 
A natural way to measure this is through
$L^2$-averages in the translation and dilation parameters of the
discrepancy function
\[
D_N(x,t)
:=\sum_{j=1}^N \chi_{-x+t\Omega}(z_j)-N\,t^d|\Omega|,
\qquad (x,t)\in\mathbb{T}^d\times(0,1].
\]
In geometric discrepancy theory, two complementary goals emerge. On the one
hand, one seeks to construct points  $\mathcal P_N$ for which the
discrepancy $D_N(x,t)$ is small in a suitable sense. On the other hand, one
aims to understand the obstructions that prevent this quantity from being too
small, thereby establishing lower bounds that hold for all configurations of
points.
See, for instance, the monograph \cite{BeckChen1987}.

A powerful approach to the study of these $L^2$-averages is provided by Fourier analysis. In many classical situations, their size is closely related to decay properties of the Fourier transform $\widehat{\chi_\Omega}$, which reflect geometric features of the set $\Omega$.

Our starting point is the following result of Brandolini and Travaglini
\cite{BT2022} in dimension $d=2$.
\begin{thm}
\label{thm:BT}
Let $\Omega$ be a convex planar body which is not a polygon and has piecewise
$\mathcal{C}^2$ boundary. Then there exists $c>0$  such that for every finite sequence $\mathcal{P}_N$ of $N$ points  in $ \mathbb{T}^2$ we
have
\[
\int_0^1\int_{\mathbb{T}^2} \left|D_N(x,t) \right|^2 dx dt
\geqslant c N^{2/5}.
\]
\end{thm}

The exponent in Theorem~\ref{thm:BT} is sharp: in \cite{BT2022} the authors
also construct a convex planar body and a point set whose averaged discrepancy is of order
$N^{2/5}$. Both the lower bound and the matching example rely on precise
estimates for $\widehat{\chi_\Omega}$, and in particular on the fact that the
Fourier transform exhibits {different decay rates in different
directions}. Exploiting related ideas, Bilyk and Mastrianni \cite{BM2023}
showed that the discrepancy of a square in $\mathbb{T}^2$, averaged over a
nontrivial interval of rotations, displays the same $N^{2/5}$ behavior.
More recently, Beretti \cite{Beretti2024} studied the affine quadratic
discrepancy of general planar convex bodies and obtained asymptotic results for the
corresponding quantities. He observed that discrepancy estimates of order $N^{2/5}$ are stable under averaging over the angular variable when the interval of rotations is sufficiently small.

To the best of the authors' knowledge, the exponent $N^{2/5}$ is the first
instance in the discrepancy literature of an averaged lower bound with this
growth rate. It should be compared with the exponent $N^{1/2}$, that is
typical for planar convex bodies with $\mathcal{C}^2$ boundary, see \cite[Theorem 4]{BT2022}, and with the logarithmic behavior
$\log N$, that is characteristic of convex polygons, see \cite{Drmota}.

In this paper we provide a $d$-dimensional analogue of
Theorem~\ref{thm:BT} where a directionally non-uniform decay for $\widehat{\chi_\Omega}$ (in an
appropriate averaged sense) leads to a discrepancy which is intermediate compared with the $N^{1-1/d}$ growth typical of smooth convex bodies, see e.g.  \cite{BCT2026}, and the $(\log N) ^{d-1}$ growth associated to convex polyhedra, see again \cite{Drmota}. 
In this context, see also the papers by Beck and Chen \cite{BeckChen1986,BeckChen1990} and by Matou\v{s}ek \cite{Mat} where the case of Cartesian products is considered. In particular the results of Beck and Chen show that the $L^2$-averaged discrepancy of the Cartesian product of a $(d-1)$-dimensional ball and an interval (a cylinder) has growth of order $N^{1-1/(d-1)}$. 

The results in \cite{BT2022} are inherently two-dimensional and exploit the fact that, in this setting, for a direction $\Theta$, one can control the Fourier transform $\widehat{\chi_\Omega}(\rho\Theta)$ in terms of chords of $\Omega$ in the direction orthogonal to $\Theta$, with the relevant contribution coming from scales of order $\rho^{-1}$ near the boundary. 

This method does not extend to higher dimensions, where no analogous control in terms of chords is available in such  generality. The main result of this type available in higher dimensions is due to Bruna, Nagel, and Wainger \cite{BNW1988}. It provides decay estimates for $\widehat{\chi_\Omega}(\rho\Theta)$ in terms of the measure of the sections of $\Omega$ by hyperplanes orthogonal to $\Theta$. However, in contrast with the planar case, this result requires a smooth boundary with finite order of contact with tangent lines. We emphasize that the estimates in \cite{BNW1988} are {upper bounds} and are not
asymptotic in nature; in particular, they do not yield matching lower
bounds for $\widehat{\chi_\Omega}$. Establishing nontrivial and uniform
lower bounds requires additional arguments and is a more delicate
problem, see for instance \cite{BCT2026}.

Motivated by the above considerations, we establish general results which provide matching lower and upper bounds for the discrepancy under suitable assumptions on the directional decay of $\widehat{\chi_\Omega}$. More precisely, we consider a situation in which the Fourier transform exhibits different decay rates in a distinguished set of directions and in its complement, and we show that this anisotropic behavior leads to matching bounds for the discrepancy.

We first fix some notation. Let $\Omega \subset \mathbb{R}^d$ be a bounded measurable set. For each $\xi \in \mathbb{R}^d$, we define the Fourier transform of the indicator function of $\Omega$ by
\[
\widehat{\chi_\Omega}(\xi)
:= \int_{\Omega} e^{-2\pi i\, \xi \cdot x}\, dx .
\]
Throughout, for $\xi \in \mathbb{R}^d$ we will write 
\[
\xi = (\xi_1,\dots,\xi_d) = (\xi',\xi_d),
\]
where $\xi' \in \mathbb{R}^{d-1}$. The same notation will also be used for $\Theta\in\Sigma_{d-1}$, the unit sphere of $\mathbb{R}^{d}$. We will also identify the unit cube $\left[-\frac{1}{2},\frac{1}{2}\right)^{d}$ with the $d$-dimensional torus $\mathbb{T}^d$. Finally, we denote by $|F|$ the Lebesgue measure of a measurable set $F\subseteq\mathbb R^d$, and by $d\sigma$ the normalized Haar measure on $SO(d)$.

Our first result is the following.

\begin{thm}\label{thm:main}
Let $\Omega\subseteq\left[-\frac{1}{2},\frac{1}{2}\right)^{d}$.
Assume there exist positive constants $\kappa_1$, $\beta_1$, and $v\in\Sigma_{d-1}$ such that for $A$ sufficiently large and for every  $\Theta\in\Sigma_{d-1}$, we have
\begin{equation}
\frac{1}{A}\int_A^{2A}|\widehat{\chi_{\Omega}}(\rho\Theta)|^2\, d\rho\geqslant\begin{cases}
\frac{\kappa_1}{A^{d+1}} \qquad |\Theta\cdot v|\leqslant \beta_1 \\
 \frac{\kappa_1}{A^{d+2}}\qquad \beta_1<|\Theta\cdot v|.
\end{cases}
\label{eq:ipotesi thm1}
\end{equation}
Then there exists a positive constant $c>0$ such that for every integer $N$ sufficiently large and for every finite
sequence of points $z_{1},z_{2},\ldots,z_{N}\in\mathbb{T}^{d}$ we have
\[
\int_{0}^{1}\int_{\mathbb{T}^{d}}\left|D_N(x,t)\right|^{2}dxdt\geqslant c N^{1-\frac{d+1}{d^2+d-1}}.
\]
\end{thm}

\begin{rem} 
Since the Fourier transform of the characteristic function of a measurable set typically exhibits an oscillatory behavior and may vanish at certain points, the averaging condition in assumption \eqref{eq:ipotesi thm1} is essential to obtain meaningful lower bounds. Consequently, an average over the dilation parameter is also required in the discrepancy estimate.
\end{rem}
 Observe that the exponent $1-\frac{d+1}{d^2+d-1}$ reduces to $\frac{2}{5}$ when $d=2$, thus matching Theorem~\ref{thm:BT}.
Moreover, since
\[
1-\frac{d+1}{d^2+d-1} < 1-\frac{1}{d},
\]
our bound is intermediate compared with the above mentioned $N^{1-1/d}$ growth  of smooth convex bodies, and the $(\log N) ^{d-1}$ growth of convex polyhedra. It is also larger than the exponent obtained in
\cite{BeckChen1986,BeckChen1990} for cylinders, suggesting that the
higher-dimensional setting allows for a range of interesting phenomena
that remain only partially understood. 

The next result shows that, under the assumption \eqref{eq:ipotesi thm1}, the above estimate for the discrepancy is optimal. Since the averaging over dilations does not play a role in this result, we set $D_N(x)=D_N(x,1)$ for convenience.

\begin{thm}\label{thm:main2}
Let $\Omega\subseteq\left[-\frac{1}{2},\frac{1}{2}\right)^{d}$. Assume there exist positive constants $\kappa_2$ and $\beta_2$ and $v\in\Sigma_{d-1}$ such that for $\rho$ sufficiently large and for every $\Theta\in\Sigma_{d-1}$ we have
\begin{align}
\left|\widehat{\chi_\Omega}(\rho\Theta)\right|^2 \leqslant\begin{cases}
        \frac{\kappa_2}{\rho^{d+1}}& |\Theta\cdot v|\leqslant\beta_2, \\
        \frac{\kappa_2}{\rho^{d+2}}& \beta_2<|\Theta\cdot v|, 
    \label{eq:ipotesi thm1 da sopra}
    \end{cases}
\end{align} 
then there exist a constant $c>0$ and a diverging sequence of integers $N_j$ such that for every $j$ there exists a finite sequence of points $z_1,\ldots,z_{N_j}\in\mathbb{T}^d$ for which
\begin{align*}
\int_{\mathbb{T}^{d}}\left|D_{N_j}(x)\right|^{2}dx\leqslant c N_j^{1-\frac{d+1}{d^2+d-1}}.  
\end{align*}
\end{thm}

The assumptions in \eqref{eq:ipotesi thm1 da sopra} may be too restrictive, especially when $\Omega$ has flat regions, as the Fourier transform may decay too slowly in the corresponding directions to meet the required estimates.
This issue can be mitigated by introducing an additional averaging over rotations in the spirit of the works of Bilyk and Mastrianni \cite{BM2023} and Beretti \cite{Beretti2024}. Let us define
\[
D_N(x,\sigma)
:=\sum_{j=1}^N \chi_{-x+\sigma\Omega}(z_j)-N|\Omega|,
\qquad (x,\sigma)\in\mathbb{T}^d\times SO(d).
\]

\begin{thm}\label{thm:main2rot}
Let $\Omega\subseteq\left[-\frac{1}{2},\frac{1}{2}\right)^{d}$. Assume there exist positive constants $\kappa_2$ and $\beta_2$, a neighborhood $U$ of the identity of $SO(d)$, and $v\in\Sigma_{d-1}$ such that for $\rho$ sufficiently large and for every $\Theta\in\Sigma_{d-1}$ we have
\begin{align}
\int_U\left|\widehat{\chi_{\sigma\Omega}}(\rho\Theta)\right|^2d\sigma \leqslant\begin{cases}
        \frac{\kappa_2}{\rho^{d+1}}& |\Theta\cdot v|\leqslant\beta_2, \\
        \frac{\kappa_2}{\rho^{d+2}}& \beta_2<|\Theta\cdot v|, 
    \label{eq:ipotesi thm da sopra rotaz}
    \end{cases}
\end{align} 
then there exist a constant $c>0$ and a diverging sequence of integers $N_j$ such that for every $j$ there exist points $z_1,\ldots,z_{N_j}\in\mathbb{T}^d$ for which
\begin{align*}
\int_U\int_{\mathbb{T}^{d}}\left|D_{N_j}(x,\sigma)\right|^{2}dxd\sigma\leqslant c N_j^{1-\frac{d+1}{d^2+d-1}}.  
\end{align*}
\end{thm}

\bigskip

The paper is organized as follows. In Section \ref{sec_proof} we establish our main abstract discrepancy estimates, proving Theorem \ref{thm:main}, Theorem \ref{thm:main2}, and Theorem \ref{thm:main2rot}. In Section \ref{sec_barrel} we test these results on a particular domain, which we call the barrel, emphasizing both their effectiveness and limitations. More precisely, we state suitable estimates for the Fourier transform of the barrel which allow us to apply the abstract results of Section \ref{sec_proof} and derive the corresponding discrepancy bounds. Section \ref{sec_aux} contains several auxiliary results on oscillatory integrals and Fresnel-type functions, which are used in Section \ref{sec_fourier} to derive asymptotic formulas for the Fourier transform of the barrel. These asymptotics are then combined in Section \ref{sect:proof} to prove the Fourier estimates stated in Section \ref{sec_barrel}.

Throughout the paper, we will use the notation $A\lesssim B$ when there exists a constant $c>0$ independent of the relevant parameters, such that $A \leqslant cB$. We will write $A\approx B$ when $A\lesssim B$ and $B\lesssim A$. Throughout the paper,  constants will be denoted by \(c\). As usual,
the value of \(c\) may change from line to line.

\section{Proofs of the main results}\label{sec_proof}
In this section we prove the main results. The lower bound is obtained by combining a Fourier-analytic estimate with a Cassels--Montgomery type argument, while the upper bound follows from a suitable anisotropic lattice construction exploiting the directional decay assumptions.

The following proposition is well known and is essentially due to Cassels in dimension one and to Montgomery in higher dimensions. For a proof inspired by Siegel’s analytic argument for Minkowski’s convex body theorem see \cite[Lemma 26]{BT2022}. See also \cite{BGG} for a more general result in the setting of compact Riemannian manifolds.

\begin{prop}[Cassels--Montgomery]
\label{cassels}
Let $U$ be a neighborhood of the origin. Then there exists a positive
constant $c$ such that for every convex symmetric body $C$ in $\mathbb{R}^{d}$
and every finite sequence $z_{1},z_{2},\ldots,z_{N}\in\mathbb{T}^{d}$
we have
\[
\sum_{m\in\left(C\setminus U\right)\cap\mathbb{Z}^{d}}\left|\sum_{j=1}^{N}e^{2\pi im\cdot z_{j}}\right|^{2}\geqslant\frac{1}{2^{d}}N\left|C\right|-cN^{2}.
\]
\end{prop}

\begin{proof}[Proof of Theorem \ref{thm:main}]
Let
\begin{align*}
\widehat{D_{N}}\left(m,t\right)&= \int_{\mathbb{T}^{d}} D_{N}\left(x,t\right)  e^{-2\pi im\cdot x}dx.
\end{align*}
A standard computation shows that $\widehat{D_{N}}\left(0,t\right)=0$, and that for $m\neq0$
\[
\widehat{D_{N}}\left(m,t\right)=\sum_{j=1}^{N}e^{2\pi im\cdot z_{j}}\widehat{\chi_{t\Omega}}\left(m\right).
\]
Hence,
\[
\int_{0}^{1}\int_{\mathbb{T}^{d}}\left|D_{N}\left(x,t\right)\right|^{2}dxdt=\sum_{m\neq0}\left|\sum_{j=1}^{N}e^{2\pi im\cdot z_{j}}\right|^{2}\int_{0}^{1}\left|\widehat{\chi_{t\Omega}}\left(m\right)\right|^{2}dt.
\]
Applying \eqref{eq:ipotesi thm1}, and possibly adjusting the constant $\kappa_1$, if $\left|m\right|$ is sufficiently large we obtain
\begin{align}
\int_{0}^{1}\left|\widehat{\chi_{t\Omega}}\left(m\right)\right|^{2}dt\geqslant & \int_{1/2}^{1}\left|t^{d}\widehat{\chi_{\Omega}}\left(tm\right)\right|^{2}dt
\label{eq:controllo da sotto}
\geqslant  \frac{c}{\left|m\right|} \int_{\left|m\right|/2}^{\left|m\right|}\left|\widehat{\chi_{\Omega}}\left(\rho\frac{m}{\left|m\right|}\right)\right|^{2}d\rho \\
\geqslant & \begin{cases}
\frac{\kappa_1}{\left|m\right|^{d+1}} \qquad |m\cdot v|\leqslant \beta_1\left|m\right| \\
 \frac{\kappa_1}{\left|m\right|^{d+2}}\qquad \beta_1\left|m\right|<|m\cdot v|.
\end{cases}
\notag
\end{align}
Let  $0<Y<X$ be parameters to be chosen later. Let
\[
F_{0}=\left\{ \left(x',x_{d}\right)\in\mathbb{R}^{d}:\left|x'\right|<X, \left|x_{d}\right|<Y\right\}
\]
and for every $\omega\in\Sigma_{d-1}$  let $F_{\omega}$ be a copy of $F_{0}$
obtained by a rotation $R_\omega$ that maps $e_{d}=\left(0,\ldots,0,1\right)$
to $\omega.$ Let $\beta_0<\beta_1$, let
\[
C_{\beta_0}=\left\{ \omega\in\Sigma_{d-1}:\omega\cdot v\geqslant\sqrt{1-\beta_0^2}\right\},
\]
and for every $\delta>0$ let $T_{\delta}$ be a maximal set of $\delta$-separated points in $C_{\beta_0}$. Observe
that for sufficiently small $\delta$,
\begin{equation}
    \text{card}\left(T_{\delta}\right) \approx {\delta^{1-d}}.
\label{eq:approx_card}
\end{equation}
Indeed, if $\sigma$ denotes the measure on $\Sigma_{d-1}$ and $C(\omega,\delta)$ denotes the spherical cap centered in $\omega$ of radius $\delta$, we have,
\[
c\,\text{card}(T_\delta)\delta^{d-1}\leqslant
\sum_{\omega\in T_\delta}\sigma\left(C(\omega,\delta)\right)= 
\sigma\left(\bigcup_{\omega\in T_\delta}C(\omega,\delta)\right)\leqslant \sigma\left(\Sigma_{d-1}\right).
\]
Moreover, by the maximality of $T_\delta$, we have that 
\[
C_{\beta_0} \subseteq \bigcup_{\omega\in T_\delta} C(\omega,3\delta).
\]
Then, 
\begin{align*}
\sigma\left(C_{\beta_0}\right)\leqslant \sum_{\omega\in T_\delta}\sigma\left(C(\omega, 3\delta)\right)\leqslant c\,\text{card}(T_\delta)\delta^{d-1},
\end{align*}
so that \eqref{eq:approx_card} follows.
Let us now fix $\delta=\frac Y{2X}$ and let
\[
\Phi\left(m\right)=\sum_{\omega\in T_{\delta}}\chi_{F_{\omega}}\left(m\right).
\]
We look for a constant $\Gamma$ such that, for every $m\in\mathbb{Z}^{d}$,
we have
\begin{equation}
\Gamma\Phi\left(m\right)\leqslant\begin{cases}
\frac{\kappa_1}{\left|m\right|^{d+1}} & \left|m\cdot v \right| \leqslant\beta_1\left|m\right|,\\
\frac{\kappa_1}{\left|m\right|^{d+2}} & \beta_1\left|m\right|<\left|m\cdot v\right|.
\end{cases}
\label{eq:Fourier below}
\end{equation}
Let $\kappa_0>0$ be such that $\frac{1}{\kappa_0}+\beta_{0}<\beta_1$. We assume that $m\in\bigcup\limits_{\omega\in T_{\delta}}F_{\omega}$ otherwise there is nothing to prove and we consider the cases $\left|m\right|>\kappa_0Y$ and $\left|m\right|\leqslant \kappa_0 Y$ separately. Assume first $\left|m\right|>\kappa_0Y$. Since $m\in F_\omega$ for some $\omega\in T_\delta$ and $Y<X$ we have $\left|m\right|\leqslant \sqrt2 X$. Moreover we have
${\left|m\cdot v \right|} \leqslant\beta_1 {\left|m\right|}.$
Indeed, let $\zeta \in F_0$ be such that $m=R_\omega \zeta$ and let us write
\[ m\cdot v=\left(m\cdot\omega\right)\left(v\cdot\omega\right)+m\cdot\left[v-\left(v\cdot\omega\right)\omega\right].
\]
Then
\begin{align*}
    \left|m\cdot v\right|\leqslant & \left|m\cdot\omega\right|+\left|m\right|\left|v-\left(v\cdot\omega\right)\omega\right| \\
\leqslant	& \left|\zeta\cdot R^{-1}_{\omega}\omega\right|+\left|m\right|\sqrt{1-\left(\omega\cdot v\right)^{2}}\leqslant\left|\zeta\cdot e_{d}\right|+\left|m\right|\beta_{0} \\
\leqslant	& Y+\left|m\right|\beta_{0}\leqslant\left(\frac{1}{\kappa_{0}}+\beta_{0}\right)\left|m\right|\leqslant\beta_{1}\left|m\right|.
\end{align*}
It follows that, to have \eqref{eq:Fourier below} satisfied, we need a positive constant  $\Gamma$ such that 
\begin{equation}
\Gamma\Phi\left(m\right)\leqslant\frac{\kappa_1}{\left|m\right|^{d+1}}.
\label{Gamma_1}
\end{equation}
Since
\begin{align*}
\Phi\left(m\right)&= \sum_{\omega\in T_{\delta}}\chi_{F_{\omega}}\left(m\right)=\text{card}\left\{ \omega\in T_{\delta}:m\in F_{\omega}\right\},
\end{align*}
the idea is to  estimate
$\sigma \left(\left\{ \omega\in C_{\beta_0}:m\in F_{\omega}\right\} \right)$ and then deduce a bound for the above cardinality.
Let $\omega\in \left\{ \omega\in C_{\beta_0}:m\in F_{\omega}\right\}$, since $ m=R_\omega\zeta$, for some $\zeta \in F_0$ we have
$
\left|m\cdot\omega\right|=\left|\zeta_d\right|\leqslant Y.
$
Hence,
\[
\left\{ \omega\in C_{\beta_0}:m \in F_{\omega}\right\} \subseteq \left\{\omega \in C_{\beta_0} : \left| \frac{m}{|m|}\cdot \omega\right|\leqslant \frac{Y}{|m|}\right\}=S_{\beta_0}.
\]
Observe that $S_{\beta_0}$ is the intersection of a spherical slab of thickness $2\frac{Y}{\left|m\right|}$ with  $C_{\beta_0}$ so that $\sigma(S_{\beta_0})\leqslant c \frac{Y}{\left|m\right|}$. Also, observe that $\delta=\frac{Y}{2X}<\frac{Y}{|m|}$, so that if $\omega\in T_\delta\cap S_{\beta_0}$ then 
\[
\sigma(C(\omega,\delta)\cap S_{\beta_0})\geqslant\frac 12 \sigma(C(\omega,\delta)).
\]
Since 
\[
\bigcup_{\omega\in T_\delta\cap S_{\beta_0}}\left(C(\omega,\delta)\cap S_{\beta_0} \right)\subseteq S_{\beta_0},
\]
we have
\[
\sum_{\omega\in T_\delta\cap S_{\beta_0}} \sigma\left(C(\omega,\delta)\cap S_{\beta_0} \right)\leqslant \sigma(S_{\beta_0}).
\]
Therefore 
$
 c\delta^{d-1} \text{card} \left(T_\delta\cap S_{\beta_0}\right) \leqslant \sigma(S_{\beta_0}).$
Hence we have 
\[
\Phi\left(m\right)\leqslant\text{card}\left\{ \omega\in T_{\delta}:m\in F_{\omega}\right\} \leqslant c \frac{Y}{\left|m\right|\delta^{d-1}}=c\frac{X^{d-1}}{\left|m\right|Y^{d-2}}.
\]
It follows that \eqref{Gamma_1} is satisfied if 
\[
\Gamma\frac{c X^{d-1}}{\left|m\right|Y^{d-2}}\leqslant\frac{\kappa_1}{\left|m\right|^{d+1}}.
\]
This in turn gives
\begin{equation}
\Gamma\leqslant\inf_{\substack{\left|m\right|>\kappa_0 Y\\ m\in\bigcup\limits_{\omega\in T_{\delta}}F_{\omega}}}\frac{cY^{d-2}}{X^{d-1}\left|m\right|^{d}}=c\,\frac{Y^{d-2}}{X^{2d-1}}.
\label{eq:Gamma out}
\end{equation}
Assume now $\left|m\right|\leqslant \kappa_0 Y$. Then by \eqref{eq:approx_card}
\[
\Phi\left(m\right)\leqslant\text{card}\left(T_{\delta}\right)\leqslant c\delta^{1-d}=c\frac{X^{d-1}}{Y^{d-1}}.
\]
It follows that
\[
\Gamma\Phi\left(m\right)\leqslant\frac{\kappa_1}{\left|m\right|^{d+2}}
\]
holds whenever
\[
c\,\Gamma\frac{X^{d-1}}{Y^{d-1}}\leqslant\inf_{\left|m\right|\leqslant \kappa_0 Y}\frac{c}{\left|m\right|^{d+2}}=\frac{c}{Y^{d+2}},
\]
that is
\begin{equation}
\Gamma\leqslant\frac{c}{Y^{d+2}}\frac{Y^{d-1}}{X^{d-1}}=\frac{c}{X^{d-1}Y^{3}}.
\label{eq:Gamma_in}
\end{equation}
Combining \eqref{eq:Gamma out} and \eqref{eq:Gamma_in} we obtain
\[
\Gamma\leqslant\min\left(\frac{c}{X^{d-1}Y^{3}},\frac{cY^{d-2}}{X^{2d-1}}\right)=\frac{c}{X^{d-1}Y^3}\min\left(1,\frac{Y^{d+1}}{X^{d}}\right).
\]
This suggests $X^{d}=Y^{d+1}$.
We also require that the measure of every $F_\omega$ is a positive fraction of the number of points $N$. Namely we set 
$X^{d-1}Y= \gamma N$ with $\gamma$ to be chosen later. 
The above two conditions give
\[
X=c_{1}N^{\frac{d+1}{d^{2}+d-1}}
\]
and
\[
Y=c_{2}N^{\frac{d}{d^{2}+d-1}}.
\]
Therefore
\[
\text{card}T_{\delta}\approx c\frac{X^{d-1}}{Y^{d-1}}=cN^{\frac{d-1}{d^{2}+d-1}}
\]
and
\[
\Gamma=\frac{c}{Y^{3}X^{d-1}}=c_{3}N^{-\frac{d^{2}+3d-1}{d^{2}+d-1}}.
\]
By \eqref{eq:controllo da sotto} and \eqref{eq:Fourier below} there exists a neighborhood of the origin $U$ such that for $m\in \mathbb{Z}^d\setminus U$ we have
\[
\int_{0}^{1}\left|\widehat{\chi_{t\Omega}}\left(m\right)\right|^{2}dt\geqslant\Gamma\Phi\left(m\right).
\]
Therefore, recalling that $|F_\omega|=\gamma N$, we have
\begin{align*}
\int_{0}^{1}\int_{\mathbb{T}^{d}}\left|D_{N}\left(x,t\right)\right|^{2}dxdt \geqslant &\sum_{m\in\mathbb{Z}^d\setminus U}\left|\sum_{j=1}^{N}e^{2\pi im\cdot z_{j}}\right|^{2}\Gamma\sum_{\omega\in T_{\delta}}\chi_{F_{\omega}}\left(m\right)\\
= & \Gamma\sum_{\omega\in T_{\delta}}\sum_{m\in\mathbb{Z}^d\setminus U}\chi_{F_{\omega}}\left(m\right)\left|\sum_{j=1}^{N}e^{2\pi im\cdot z_{j}}\right|^{2}\\
\geqslant & \Gamma\sum_{\omega\in T_{\delta}}\left(\frac{1}{2^{d}}\gamma N^{2}-cN^{2}\right), 
\end{align*}
where we applied Proposition \ref{cassels} on each $F_\omega$. 
Choosing $\gamma$ large enough yields
\[
\int_{0}^{1}\int_{\mathbb{T}^{d}}\left|D_{N}\left(x,t\right)\right|^{2}dxdt\geqslant
c\,\Gamma\,\text{card}\left(T_{\delta}\right)  N^{2}\geqslant
cN^{1-\frac{d+1}{d^{2}+d-1}}.
\]
\end{proof}
We now prove the upper bound for the discrepancy via an anisotropic lattice construction exploiting \eqref{eq:ipotesi thm1 da sopra}.

\begin{proof}[Proof of Theorem \ref{thm:main2}]
Our argument requires a lattice rotated by a matrix with rational entries. We can achieve this by slightly modifying the cone in the estimate \eqref{eq:ipotesi thm1 da sopra}. More precisely, let $\sigma_0\in SO(d)$ be such that $v=\sigma_0 e_d$. Let $\beta_3\in(\beta_2,1)$ and let $0<\varepsilon<\beta_3-\beta_2$. Since $SO(d,\mathbb{Q})$ is dense in $SO(d)$ we can find  $\sigma_1\in SO(d,\mathbb{Q})$ such that $\left\Vert \sigma_{0}-\sigma_{1}\right\Vert <\varepsilon$. Let $v_1=\sigma_1 e_d$. Then
\[
\left|\widehat{\chi_{\Omega}}(\rho\Theta)\right|^{2}\leqslant\begin{cases}
\frac{\kappa_{2}}{\rho^{d+1}} & |\Theta\cdot v_{1}|\leqslant\beta_{3},\\
\frac{\kappa_{2}}{\rho^{d+2}} & \beta_{3}<|\Theta\cdot v_{1}|.
\end{cases}
\]
Indeed, the estimate in the case $|\Theta\cdot v_{1}|\leqslant\beta_{3}$ follows directly from \eqref{eq:ipotesi thm1 da sopra}. Assume $\beta_{3}<|\Theta\cdot v_{1}|$. Then 
\[ \left|\Theta\cdot v\right|\geqslant\left|\Theta\cdot v_{1}\right|-\left|\Theta\cdot\left(v-v_{1}\right)\right|>\beta_{3}-\left|\sigma_{0}e_{d}-\sigma_{1}e_{d}\right|\geqslant\beta_{3}-\varepsilon
\]
and the estimate for this second case follows again from \eqref{eq:ipotesi thm1 da sopra}.

Now, let $N_{0}$ be a positive integer, let 
\[
G=\left[N_{0}^{\frac{d+1}{d^{2}+d-1}}\right],
\]
and
\[
L=\left[N_{0}^{\frac{d}{d^{2}+d-1}}\right],
\]
where $\left[\cdot\right]$ denotes the integer parts and let $N=G^{d-1}L$. Note that for large $N_{0}$, we have $N\approx N_{0}$. Let
\[
\mathcal{P}_{N}=\left(\frac{1}{G}\mathbb{Z}^{d-1}\times\frac{1}{L}\mathbb{Z}\right)\cap\left[-\frac{1}{2},\frac{1}{2}\right)^{d},
\]
and let {$\left\{ z_{j}\right\} ^{N}_{j=1}$ }be any enumeration
of the points in $\mathcal{P}_{N}$. Let $\kappa\in\mathbb N$ be such that
$A=\kappa\sigma_1$ has integer entries and let $\widetilde{z}_{j}=Az_{j}\mod 1$. The resulting finite sequence may have repetitions, which are allowed. Identifying the $d$-dimensional
torus $\mathbb{T}^{d}$ with the cube $\left[-\frac{1}{2},\frac{1}{2}\right)^{d}$ we can
assume that $\widetilde{z}_{j}\in\mathbb{T}^{d}$.

We now estimate the discrepancy associated with the points $\widetilde{z}_j$. We have
\[
\int_{\mathbb{T}^{d}}\left|D_{N}\left(x\right)\right|^{2}dx=\sum_{m\neq0}\left|\sum_{j=1}^N e^{2\pi im\cdot \widetilde{z}_j}\right|^{2}\left|\widehat{\chi_{\Omega}}\left(m\right)\right|^{2}.
\]
Since
\[
\sum^{N}_{j=1}e^{2\pi im\cdot\widetilde{z}_{j}}=\sum^{N}_{j=1}e^{2\pi im\cdot Az_{j}}=\sum^{N}_{j=1}e^{2\pi iA^{t}m\cdot z_{j}}
\]
we have
\[
\sum^{N}_{j=1}e^{2\pi im\cdot\widetilde{z}_{j}}=\begin{cases}
G^{d-1}L & \text{if \ensuremath{A^{t}m\in G\mathbb{Z}^{d-1}\times L\mathbb{Z}}},\\
0 & \text{otherwise.}
\end{cases}
\]
Indeed, let $A^t m\in G\mathbb{Z}^{d-1}\times L\mathbb{Z}$ and observe
that if $z\in\mathcal{P}_N$ then $A^tm\cdot z\in\mathbb{Z}$. It
follows that 
\[
\sum^{N}_{j=1}e^{2\pi im\cdot\widetilde{z}_{j}}=\sum_{j=1}^N 1=G^{d-1}L.
\]
Now, let $A^t m\notin G\mathbb{Z}^{d-1}\times L\mathbb{Z}$. It follows
that there exists $z'\in\mathcal{P}_{N}$ such that $e^{2\pi i A^t m\cdot z'}\neq1$,
hence
\[
e^{2\pi i A^t m\cdot z'}\sum^{N}_{j=1}e^{2\pi i A^t m\cdot  z_{j}}=\sum^{N}_{j=1}e^{2\pi i A^t m\cdot  (z_{j}+z')}=\sum^{N}_{j=1}e^{2\pi i A^t m\cdot z_{j}}
\]
and therefore $\sum^{N}_{j=1}e^{2\pi i A^t m\cdot\widetilde{z}_{j}}=0$. 

Since $\left(A^{t}\right)^{-1}=\kappa^{-1}\sigma_{1}$ we have
$$
\mathcal{R}=	\left\{ m\in\mathbb{Z}^{d}\setminus\left\{ 0\right\} :A^{t}m\in G\mathbb{Z}^{d-1}\times L\mathbb{Z}\right\} 
=	\kappa^{-1}\sigma_{1}\left(G\mathbb{Z}^{d-1}\times L\mathbb{Z}\right)\cap\left(\mathbb{Z}^{d}\setminus\left\{ 0\right\} \right)
$$
and therefore 
\begin{align*}
 & \sum_{m\in\mathbb{Z}^{d}\setminus\left\{ 0\right\} }\left|\sum^{N}_{j=1}e^{2\pi im\cdot\widetilde{z}_{j}}\right|^{2}\left|\widehat{\chi_{\Omega}}\left(m\right)\right|^{2}= G^{2d-2}L^{2}\sum_{m\in\mathcal{R}}\left|\widehat{\chi_{\Omega}}\left(m\right)\right|^{2}\\
= & G^{2d-2}L^{2}\sum_{\kappa^{-1}\sigma_{1}\left(Gn',Ln_{d}\right)\in\mathcal{R}}\left|\widehat{\chi_{\Omega}}\left(\kappa^{-1}\sigma_{1}\left(Gn',Ln_{d}\right)\right)\right|^{2}\\
\leqslant & G^{2d-2}L^{2}\sum_{n\in\mathbb{Z}^{d}\setminus\{0\}}\left|\widehat{\chi_{\Omega}}\left(\kappa^{-1}\sigma_{1}\left(Gn',Ln_{d}\right)\right)\right|^{2}.
\end{align*}
Writing
\begin{align*}\kappa^{-1}\sigma_{1}\left(Gn',Ln_{d}\right)=\left|\left(Gn',Ln_{d}\right)\right|\kappa^{-1}\sigma_{1}\frac{\left(Gn',Ln_{d}\right)}{\left|\left(Gn',Ln_{d}\right)\right|}\end{align*}
and noting that
$\sigma_{1}\frac{\left(Gn',Ln_{d}\right)}{\left|\left(Gn',Ln_{d}\right)\right|}\cdot v_1=\frac{\left(Gn',Ln_{d}\right)}{\left|\left(Gn',Ln_{d}\right)\right|}\cdot e_{d}$,
we obtain
\[
\left|\widehat{\chi_{\Omega}}\left(\kappa^{-1}\sigma_{1}\left(Gn',Ln_{d}\right)\right)\right|^{2}\leqslant\begin{cases}
\frac{c}{\left|\left(Gn',Ln_{d}\right)\right|^{d+1}} & \text{if }\frac{\left|Ln_d\right|}{\left|\left(Gn',Ln_{d}\right)\right|}\leqslant\beta_3,\\
\frac{c}{\left|\left(Gn',Ln_{d}\right)\right|^{d+2}} & \text{if }\frac{\left|Ln_d\right|}{\left|\left(Gn',Ln_{d}\right)\right|}>\beta_3.
\end{cases}
\]
Equivalently 
\[
\left|\widehat{\chi_{\Omega}}\left(\kappa^{-1}\sigma_{1}\left(Gn',Ln_{d}\right)\right)\right|^{2}\leqslant\begin{cases}
\frac{c}{\left|\left(Gn',Ln_{d}\right)\right|^{d+1}} & \text{if }|Ln_d|\leqslant\eta|Gn'|,\\
\frac{c}{\left|\left(Gn',Ln_{d}\right)\right|^{d+2}} & \text{if }|Ln_d|>\eta|Gn'|
\end{cases}
\]
for $\eta=\frac{\beta_3}{\sqrt{1-\beta_3^2}}$. Therefore
\begin{align*}
\int_{\mathbb{T}^{d}}\left|D_{N}\left(x\right)\right|^{2}dx 
\leqslant & G^{2d-2}L^{2}\sum_{n\neq0,\left|Ln_{d}\right|\leqslant\eta\left|Gn'\right|}\frac{\kappa_{2}}{\left|Gn'\right|^{d+1}} +G^{2d-2}L^{2}\sum_{n\neq0,\left|Ln_{d}\right|>\eta\left|Gn'\right|}\frac{\kappa_{2}}{\left|Ln_{d}\right|^{d+2}}\\
= & S_{1}+S_{2}.
\end{align*}
For the first sum, since $G>L$ we have
\begin{align*}
S_{1}\lesssim & G^{d-3}L^{2}\sum_{n\neq0,\left|Ln_{d}\right|\leqslant\eta\left|Gn'\right|}\frac{1}{\left|n'\right|^{d+1}}= G^{d-3}L^{2}\sum_{n'\neq0}\frac{1}{\left|n'\right|^{d+1}}\sum_{ |n_{d}|\leqslant\eta\frac{G}{L}\left|n'\right|}1\\
\lesssim & G^{d-3}L^{2}\sum_{n'\neq0}\frac{1}{\left|n'\right|^{d+1}}\left(1+2\eta\frac{G}{L}\left|n'\right|\right)\lesssim G^{d-3}L^{2}\left(1+\frac{G}{L}\right)\lesssim  G^{d-2}L.
\end{align*}
For the second sum we have
\begin{align*}
S_{2}\lesssim & G^{2d-2}G^{-d}\sum_{n\neq0,\left|Ln_{d}\right|>\eta\left|Gn'\right|}\frac{1}{\left|n_{d}\right|^{d+2}}=G^{2d-2}L^{-d}\sum_{n_{d}\neq0}\frac{1}{\left|n_{d}\right|^{d+2}}\sum_{\left|n'\right|<\eta^{-1}\frac{L}{G}\left|n_{d}\right|}1\\
\lesssim & G^{2d-2}L^{-d}\sum_{n_{d}\neq0}\frac{1}{\left|n_{d}\right|^{d+2}}\left(1+\left|\eta^{-1}\frac{L}{G}n_{d}\right|^{d-1}\right)
\lesssim  G^{2d-2}L^{-d}\left(1+\left(\frac{L}{G}\right)^{d-1}\right)
\lesssim  G^{2d-2}L^{-d}.
\end{align*}
Consequently,
\begin{align*}
\int_{\mathbb{T}^{d}}\left|D_{N}\left(x\right)\right|^{2}dx\lesssim & G^{d-2}L+G^{2d-2}L^{-d}\lesssim N_{0}^{1-\frac{d+1}{d^{2}+d-1}}\lesssim N^{1-\frac{d+1}{d^{2}+d-1}}.
\end{align*}

\end{proof}

\begin{proof}[Proof of Theorem \ref{thm:main2rot}]
Since
\[
\int_U\int_{\mathbb{T}^{d}}\left|D_{N}\left(x,\sigma\right)\right|^{2}dxd\sigma=\sum_{m\neq0}\left|\sum_{j=1}^N e^{2\pi im\cdot{z}_j}\right|^{2}\int_U\left|\widehat{\chi_{\sigma\Omega}}\left(m\right)\right|^{2}d\sigma
\]
the proof follows the same argument as that of Theorem \ref{thm:main2},  with the pointwise bounds replaced by their averaged counterparts.
\end{proof}

\section{Discrepancy estimates for the barrel}\label{sec_barrel}
In this section, we apply our results to a particular convex body, which we refer to as the {barrel}, that exhibits all the geometric features known to play a key role in discrepancy theory: flat regions, curved regions, and edges.  
Let $0<\alpha<1$ and $0<\beta<\sqrt{1-\alpha^{2}}$. We define the barrel $\mathcal{B}=\mathcal{B}_{\alpha,\beta}$ by
\[
\mathcal{B}_{\alpha,\beta}=\left\{ \left(x,z\right)\in\mathbb{R}^{d-1}\times\mathbb{R}:\left(\left|x\right|+\alpha\right)^{2}+z^{2}\leqslant1,\left|z\right|\leqslant\beta\right\} .
\]

We first establish lower bounds on the averages of $\widehat{\chi_\mathcal{B}}$.

\begin{thm}
\label{thm:0}There exist positive constants $\varepsilon$ and $c$ such that, for all sufficiently large $A$, the following estimates hold.
\begin{enumerate}
\item 
Let $\left|\Theta_{d}\right|\leqslant\beta$, then
\[
\frac{1}{A}\int_{A}^{2A}\left|\widehat{\chi_{\mathcal{B}}}\left(\rho\Theta\right)\right|^{2}d\rho\geqslant\frac{c}{A^{d+1}}.
\]
\item Let $\beta<\left|\Theta_{d}\right|\leqslant\beta+\varepsilon$,
then
\begin{align*}
\frac{1}{A}\int_{A}^{2A}\left|\widehat{\chi_{\mathcal{B}}}\left(\rho\Theta\right)\right|^{2}d\rho & \geqslant\frac{c}{A^{d+1}}\frac{1}{1+\left(A^{1/2}\left(\left|\Theta_{d}\right|-\beta\right)\right)^{2}}.
\end{align*}
\item Let $\beta+\varepsilon<\left|\Theta_{d}\right|\leqslant1$, then
\[
\frac{1}{A}\int_{A}^{2A}\left|\widehat{\chi_{\mathcal{B}}}\left(\rho\Theta\right)\right|^{2}d\rho\geqslant\frac{c}{A^{2}}\frac{1}{1+\left(A\left|\Theta'\right|\right)^{d}}.
\]
\end{enumerate}
\end{thm}
We postpone the proof of this theorem to Section \ref{sect:proof}. From these estimates we deduce the following lower bound for the discrepancy.

\begin{cor}
Let $\mathcal{B}_0=t_0\mathcal{B}$, where $t_0$ is sufficiently small so  that $\mathcal{B}_0\subseteq\left[-\frac12,\frac12\right)^d$. Then
there exists a positive constant $c$ such that for every sufficiently large integer $N$ and for every finite
sequence of points $z_{1},z_{2},\ldots,z_{N}\in\mathbb{T}^{d}$ we have
\[
\int_{0}^{1}\int_{\mathbb{T}^{d}}\left|\sum_{j=1}^{N}\chi_{-x+t\mathcal{B}_0}\left(z_{j}\right)-Nt^{d}\left|\mathcal{B}_0\right|\right|^{2}dxdt\geqslant c N^{1-\frac{d+1}{d^2+d-1}}.
\]
\end{cor}
\begin{proof}
    We first observe that $\mathcal{B}_0$ satisfies the same estimates as in Theorem \ref{thm:0}. These estimates, in turn, readily imply assumption \eqref{eq:ipotesi thm1} and hence the estimate for the discrepancy given by Theorem \ref{thm:main}.
\end{proof}

\begin{rem}
  As already observed, the exponent
\[
1-\frac{d+1}{d^2+d-1}
\]
lies between the exponent \(1-\frac{1}{d-1}\), associated with cylinders,
and the exponent \(1-\frac{1}{d}\), associated with balls. The distinction
between these two classical regimes is essentially governed by curvature:
all principal curvatures are non-zero for the ball, while cylinders have
flat directions. The barrel combines curvature and singularities, since it
has curved regions with non-vanishing principal curvatures as well as
edges. The presence of these edges therefore interacts with curvature and
leads to a further intermediate regime. This illustrates once again that
higher-dimensional non-smooth convex bodies may exhibit subtle discrepancy
phenomena that are still only partially explored.
\end{rem}

In Section \ref{sect:proof}, we also establish the following pointwise upper bounds for $\widehat{\chi_\mathcal{B}}$.
\begin{thm}
\label{thm:above}There exist positive constants $\varepsilon$ and $c$
such that for $\rho$ sufficiently large we have the following estimates.
\begin{enumerate}
\item 
Let $\left|\Theta_{d}\right|\leqslant\beta+\varepsilon$, then
\[
\left|\widehat{\chi_{\mathcal{B}}}\left(\rho\Theta\right)\right|^{2}\leqslant\frac{c}{\rho^{d+1}}.
\]
\item Let $\beta+\varepsilon<\left|\Theta_{d}\right|\leqslant1$, then
\[
\left|\widehat{\chi_{\mathcal{B}}}\left(\rho\Theta\right)\right|^{2}\leqslant\frac{c}{\rho^{2}}\frac{1}{1+\left(\rho\left|\Theta'\right|\right)^{d}}.
\]
\end{enumerate}
\end{thm}

Due to the slow decay in the directions where $|\Theta_d|=1$, we cannot  apply Theorem \ref{thm:main2}. As anticipated in the introduction, we overcome this issue by averaging $\left|\widehat{\chi_{\mathcal{B}}}\left(\rho\Theta\right)\right|^{2}$ over a small family of rotations. We start with a lemma. 

\begin{lem}
\label{lem:rotaz}Let $\Omega\subset\mathbb{R}^{d}$ be a bounded
measurable set. Assume that there exists $\delta>0$ such that, for
all $\rho>1$ and $\left|\Theta\right|=1$,
\begin{align*}
\left|\widehat{\chi_{\Omega}}\left(\rho\Theta\right)\right|^{2} & \leqslant\begin{cases}
\frac{c}{\rho^{2}}\frac{1}{1+\left|\rho\Theta'\right|^{d}} & \text{if }\left|\Theta'\right|<\delta,\\
\frac{c}{\rho^{d+1}} & \text{if }\left|\Theta'\right|\geqslant\delta.
\end{cases}
\end{align*}
Then, there exists a symmetric neighborhood $U$ of the identity in
$SO\left(d\right)$ such that
\begin{equation}
\int_{U}\left|\widehat{\chi_{\sigma\Omega}}\left(\rho\Theta\right)\right|^{2}d\sigma\leqslant\begin{cases}
\frac{c}{\rho^{d+1}} & \left|\Theta'\right|\leqslant\delta/4,\\
\frac{c}{\rho^{d+2}} & \delta/4<\left|\Theta'\right|\leqslant\delta/2,\\
\frac{c}{\rho^{d+1}} & \left|\Theta'\right|>\delta/2.
\end{cases}\label{eq:tesi lemma}
\end{equation}
\end{lem}

\begin{proof}
Let $U$ be a symmetric neighborhood of the identity of $SO\left(d\right)$,
sufficiently small so that for all $\sigma\in U$ and for all $|\Theta|=1$
we have $|(\sigma^{-1}\Theta)'-\Theta'|<\frac{\delta}{8}$. This choice
implies that for $\sigma\in U$ and $|\Theta'|\leqslant\delta/2$ we
have 
\[
|(\sigma^{-1}\Theta)'|\leqslant|\Theta'|+|(\sigma^{-1}\Theta)'-\Theta'|<\frac{\delta}{2}+\frac{\delta}{8}<\delta
\]
and for $\sigma\in U$ and $\left|\Theta'\right|\geqslant\delta/4$
we have
\begin{equation}
\left|(\sigma^{-1}\Theta)'\right|\geqslant|\Theta'|-|(\sigma^{-1}\Theta)'-\Theta'|>\frac{\delta}{4}-\frac{\delta}{8}=\frac{\delta}{8}.\label{eq:delta ottavi}
\end{equation}
Let $\left|\Theta'\right|\leqslant\delta/4$ and let $\sigma_{\Theta}\in SO\left(d\right)$
be such that $\sigma_{\Theta}e_{d}=\Theta$. Since
$
\widehat{\chi_{\sigma\Omega}}\left(\rho\Theta\right)=\widehat{\chi_{\Omega}}\left(\rho\sigma^{-1}\Theta\right),
$
we have
\begin{align*}
\int_{U}\left|\widehat{\chi_{\sigma\Omega}}\left(\rho\Theta\right)\right|^{2}d\sigma\leqslant & \frac{c}{\rho^{2}}\int_{SO\left(d\right)}\frac{1}{1+\rho^{d}\left|\left(\sigma^{-1}\Theta\right)'\right|^{d}}d\sigma\leqslant  \frac{c}{\rho^{2}}\int_{SO\left(d\right)}\frac{1}{1+\rho^{d}\left|\left(\sigma\sigma_{\Theta}e_{d}\right)'\right|^{d}}d\sigma\\
= & \frac{c}{\rho^{2}}\int_{SO\left(d\right)}\frac{1}{1+\rho^{d}\left|\left(\sigma e_{d}\right)'\right|^{d}}d\sigma=  \frac{c}{\rho^{2}}\int_{\Sigma_{d-1}}\frac{1}{1+\left(\rho\left|x'\right|\right)^{d}}dx,
\end{align*}
where the last identity follows from the rotational invariance of
the Haar measure, since $\sigma\mapsto\sigma e_{d}$ pushes it forward
to the surface measure on $\Sigma_{d-1}$. To integrate over $\Sigma_{d-1}$,
we write every $x\in\Sigma_{d-1}$ as 
\[
x=\left(x',x_{d}\right)=\left(\sin\left(\phi\right)\eta,\cos\left(\phi\right)\right)
\]
 with $\eta\in\Sigma_{d-2}$ and $\phi\in\left[0,\pi\right]$. The
surface measure satisfies 
\[
dx=\sin\left(\phi\right)^{d-2}d\phi d\eta.
\]
 Hence
\begin{align*}
\int_{\Sigma_{d-1}}\frac{1}{1+\left(\rho\left|x'\right|\right)^{d}}dx= & 2|\Sigma_{d-2}| \int^{\pi/2}_{0}\frac{\sin\left(\phi\right)^{d-2}}{1+\left(\rho\sin\left(\phi\right)\right)^{d}}d\phi\\
\leqslant & c\int^{1/\rho}_{0}\sin\left(\phi\right)^{d-2}d\phi+c\int^{\pi/2}_{1/\rho}\frac{\sin\left(\phi\right)^{d-2}}{\left(\rho\sin\left(\phi\right)\right)^{d}}d\phi \leqslant\frac{c}{\rho^{d-1}}
\end{align*}
which gives the first inequality in (\ref{eq:tesi lemma}). Assume now
$\frac{1}{4}\delta<\left|\Theta'\right|<\frac{1}{2}\delta$, then,
by (\ref{eq:delta ottavi}),
\begin{align*}
\int_{U}\left|\widehat{\chi_{\sigma\Omega}}\left(\rho\Theta\right)\right|^{2}d\sigma\leqslant & \frac{c}{\rho^{2}}\int_{U}\frac{1}{1+\rho^{d}\left|\left(\sigma^{-1}\Theta\right)'\right|^{d}}d\sigma\leqslant \frac{c}{\rho^{2}}\int_{SO\left(d\right)}\frac{1}{\rho^{d}\left(\delta/8\right)^{d}}d\sigma\leqslant\frac{c}{\rho^{d+2}}.
\end{align*}
Finally, let $\left|\Theta'\right|>\delta/2$ and let $\sigma\in U$.
If $\left|\left(\sigma^{-1}\Theta\right)'\right|\geqslant\delta$ then
\[
\left|\widehat{\chi_{\Omega}}\left(\rho\sigma^{-1}\Theta\right)\right|^{2}\leqslant\frac{c}{\rho^{d+1}}.
\]
If $\left|\left(\sigma^{-1}\Theta\right)'\right|<\delta$ since $\left|\left(\sigma^{-1}\Theta\right)'\right|>\delta/8$
we have
\[
\left|\widehat{\chi_{\Omega}}\left(\rho\sigma^{-1}\Theta\right)\right|^{2}\leqslant\frac{c}{\rho^{2}}\frac{1}{1+\left|\rho\left(\sigma^{-1}\Theta\right)'\right|^{d}}\leqslant\frac{c}{\rho^{2}}\frac{1}{\rho^{d}\left|\delta/8\right|^{d}}\leqslant\frac{c}{\rho^{d+2}}\leqslant\frac{c}{\rho^{d+1}}.
\]
In both cases for every $\sigma\in U$ we obtain
\[
\left|\widehat{\chi_{\Omega}}\left(\rho\sigma^{-1}\Theta\right)\right|^{2}\leqslant\frac{c}{\rho^{d+1}}
\]
and the third inequality in (\ref{eq:tesi lemma}) follows. 
\end{proof}

We are now able to identify a cone for which the average over rotations satisfies the estimates from above required by Theorem \ref{thm:main2rot}. 

\begin{lem}
\label{lem:direzione}Let $\Omega$, $U$, and $\delta$ be as in
Lemma \ref{lem:rotaz}. Let $v\in\Sigma_{d-1}$ be such that 
\[
\left|v'\right|=\frac{3}{8}\delta
\]
and $v_{d}>0$. Then 
\[
\int_{U}\left|\widehat{\chi_{\sigma\Omega}}\left(\rho\Theta\right)\right|^{2}d\sigma\leqslant\begin{cases}
\frac{c}{\rho^{d+1}} & \text{if }\left|\Theta\cdot v\right|\leqslant1-\frac{1}{128}\delta^{2},\\
\frac{c}{\rho^{d+2}} & \text{if }\left|\Theta\cdot v\right|>1-\frac{1}{128}\delta^{2}.
\end{cases}
\]
\end{lem}

\begin{proof}
First of all observe that by Lemma \ref{lem:rotaz} the estimate $\frac{c}{\rho^{d+1}}$
holds for every $\Theta\in\Sigma_{d-1}$ and in particular for $\left|\Theta\cdot v\right|\leqslant1-\frac{1}{128}\delta^{2}$.
Assume now $\left|\Theta\cdot v\right|>1-\frac{1}{128}\delta^{2}$. Since the estimates in \eqref{eq:tesi lemma} are invariant under a change
of sign of \(\Theta\), it suffices to consider the case \(\Theta\cdot v>0\).
Since
\[
\left|\Theta-v\right|^{2}=2-2\Theta\cdot v\leqslant\frac{1}{64}\delta^{2},
\]
we have
\[
\left|\Theta'\right|\leqslant\left|v'\right|+\left|\Theta'-v'\right|\leqslant\frac{3}{8}\delta+\frac{1}{8}\delta=\frac{1}{2}\delta
\]
and
\[
\left|\Theta'\right|\geqslant\left|v'\right|-\left|\Theta'-v'\right|\geqslant\frac{3}{8}\delta-\frac{1}{8}\delta=\frac{1}{4}\delta.
\]
The desired estimate now follows from Lemma \ref{lem:rotaz} in the
intermediate regime.
\end{proof}

For every $(x,\sigma)\in\mathbb{T}^{d}\times SO\left(d\right)$ 
define the  discrepancy
\[
D_{N}(x,\sigma):=\sum^{N}_{j=1}\chi_{-x+\sigma\mathcal{B}}({z}_{j})-N\left|\mathcal{B}\right|.
\]

\begin{prop}
There exist a symmetric neighborhood $U$ of the identity of $SO(d)$,
a constant $c>0$, and a diverging sequence of integers $N_{j}$ such
that for every $j$ there exist points
${z}_{1},{z}_{2},\ldots,{z}_{N_{j}}$ such that
\begin{align*}
\int_{U}\int_{\mathbb{T}^{d}}\left|D_{N_{j}}(x,\sigma)\right|^{2}dxd\sigma\leqslant cN^{1-\frac{d+1}{d^{2}+d-1}}_{j}.
\end{align*}
\end{prop}

\begin{proof}
The proposition follows from Theorem \ref{thm:main2rot} combining Theorem \ref{thm:above}, Lemma \ref{lem:rotaz} and Lemma \ref{lem:direzione}.
\end{proof}

\begin{rem}
By the previous proposition, for every $j$ there exists $\sigma_{j}\in SO\left(d\right)$ and points
${z}_{1}$,${z}_{2}$, $\ldots$, ${z}_{N_{j}}$ such
that
\[
\int_{\mathbb{T}^{d}}\left|\sum^{N}_{k=1}\chi_{-x+\sigma_{j}\mathcal{B}}({z}_{k})-N_{j}\left|\mathcal{B}\right|\right|^{2}dx\leqslant cN^{1-\frac{d+1}{d^{2}+d-1}}_{j}.
\]
\end{rem}

To conclude this section, we also present a result in dimension $d=2$. We show that we can get an almost optimal upper bound (up to a logarithmic factor) for the $L^2$-average of the discrepancy even without averaging over rotations, by taking a slightly tilted barrel. The proof is in the spirit of Davenport \cite{Davenport} (see also \cite[Theorem 8.24]{GT}). The same technique does not seem to extend to higher dimensions. 

\begin{prop}
    Let $\varphi$ be a sufficiently small positive irrational algebraic number of degree $2$, let  $\sigma_{\varphi}$ be a rotation of angle $\arctan \varphi$, and let $\mathcal{B}_{\varphi}=\sigma_{\varphi}\mathcal{B}$. Then there exist a positive constant $c$ and diverging sequence of integers $N_j$ such that for every $j$ there exist points $z_1,z_2,\cdots,z_{N_j}\in\mathbb{T}^2$ such that
    \[ \int_{\mathbb{T}^{2}}\left|\sum^{N_j}_{k=1}\chi_{-x+\mathcal{B}_\varphi}({z}_{k})-N_{j}\left|\mathcal{B}\right|\right|^{2}dx
    \leqslant c N_j^{2/5}\log N_j.
    \]
\end{prop}
\begin{proof}
Since $\varphi$ is an irrational algebraic number of degree $2$ by \cite[Theorem 5.6]{GT} there exists $H>0$ such that for every integer $n\neq0$
\begin{align}
\label{eq:irr_alg}
\left\Vert n\varphi\right\Vert >\frac{H}{|n|},    
\end{align}
where $\left\Vert x\right\Vert =\min\left(\left\{ x\right\} ,1-\left\{ x\right\} \right)$
denotes the distance from the closest integer. Let $e_{\varphi}=\left(1,\varphi\right)$, and observe that $\sigma_{\varphi}$  maps $e_{1}$ to $\frac{e_{\varphi}}{\left|e_{\varphi}\right|}$. Let $e_\varphi^\perp=(-\varphi,1)$. Then by Theorem \ref{thm:above} there exists $\beta'$ such that
\[
\left|\widehat{\chi_{\mathcal{B}_{\varphi}}}\left(m\right)\right|^{2}=\left|\widehat{\chi_{\mathcal{B}}}\left(\sigma^{-1}_{\varphi}m\right)\right|^{2}\leqslant c\begin{cases}
\frac{1}{\left|m\right|^{3}} & \left|m\cdot e^{\perp}_{\varphi} \right|\leqslant\beta'\left|m\right|,\\
\frac{1}{\left|m\right|^{2}\left(1+\left|m\cdot e_{\varphi} \right|\right)^{2}} & \left|m\cdot e^{\perp}_{\varphi} \right|>\beta'\left|m\right|.
\end{cases}
\]
If $\varphi$ is sufficiently small, the above estimate implies
\[
\left|\widehat{\chi_{\mathcal{B}_{\varphi}}}\left(m\right)\right|^{2}\leqslant c\begin{cases}
\frac{1}{\left|m\right|^{3}} & \left|m_{2}\right|\leqslant\gamma\left|m_{1}\right|,\\
\frac{1}{\left|m\right|^{2}\left(1+\left| m\cdot e_{\varphi}\right|\right)^{2}} & \left|m_{2}\right|>\gamma\left|m_{1}\right|.
\end{cases}
\]
where $\gamma=\frac{\varphi+\beta'}{1-\beta'}$. Indeed, assume $\left|m_{2}\right|\leqslant\frac{\varphi+\beta'}{1-\beta'}\left|m_{1}\right|$.
If $\left|m\cdot e^{\perp}_{\varphi} \right|\leqslant\beta'\left|m\right|$
there is nothing to prove. If $\left| m\cdot e^{\perp}_{\varphi} \right|>\beta'\left|m\right|$,
since 
\[
\left| m\cdot e_{\varphi}\right|=\left|m_{1}+\varphi m_{2}\right|\geqslant\left|m_{1}\right|-\varphi\left|m_{2}\right|\geqslant\left(1-\varphi\frac{\varphi+\beta'}{1-\beta'}\right)\left|m_{1}\right|\approx\left|m\right|,
\]
we have
\[
\left|\widehat{\chi_{\mathcal{B}_{\varphi}}}\left(m\right)\right|^{2}\leqslant\frac{c}{\left|m\right|^{4}}\leqslant\frac{c}{\left|m\right|^{3}}.
\]
Finally assume $\left|m_{2}\right|>\frac{\varphi+\beta'}{1-\beta'}\left|m_{1}\right|$,
then
\begin{align*}
\left| m\cdot e^{\perp}_{\varphi} \right| & =\left|\varphi m_{1}-m_{2}\right|\geqslant\left|m_{2}\right|-\varphi\left|m_{1}\right|\geqslant\left|m_{2}\right|-\frac{\varphi\left(1-\beta'\right)}{\varphi+\beta'}\left|m_{2}\right|\geqslant\beta'\left|m\right|.
\end{align*}
Now let $L,K\in\mathbb{N}$ be such that $L/K\in\mathbb{N}$ and assume $L\gg K$.
Let us consider the points
\[
z_{\ell,k}=\left(\frac{\ell}{L},\frac{k}{K}\right)
\]
for $\ell=0,\ldots,L-1$ and $k=0,\ldots,K-1$. Then, arguing as in the proof of Theorem $\ref{thm:main2}$, 
\begin{align*}
\int_{\mathbb{T}^2} \left| D_{N_j}(x)\right|^2dx = & \sum_{m\neq0}\left|\sum^{L-1}_{\ell=0}\sum^{K-1}_{k=0}e^{2\pi im\cdot z_{\ell,k}}\right|^{2}\left|\widehat{\chi_{\mathcal{B}_{\varphi}}}\left(m\right)\right|^{2}=  L^{2}K^{2}\sum_{m\neq0}\left|\widehat{\chi_{\mathcal{B}_{\varphi}}}\left(Lm_{1},Km_{2}\right)\right|^{2}\\
\leqslant & L^{2}K^{2}\sum_{m\neq0,\left|Km_{2}\right|\leqslant\gamma\left|Lm_{1}\right|}\frac{c}{\left|Lm_{1}\right|^{3}}\\
 & +L^{2}K^{2}\sum_{m\neq0,\left|Km_{2}\right|>\gamma\left|Lm_{1}\right|}\frac{c}{\left|Km_{2}\right|^{2}\left(1+\left|\left(Lm_{1},Km_{2}\right)\cdot e_{\varphi} \right|\right)^{2}}\\
= & S_{1}+S_{2}.
\end{align*}
For the first term we have 
\begin{align*}
S_{1}= & \frac{K^{2}}{L}\sum_{\left|m_{1}\right|\geqslant 1}\sum_{\left|m_{2}\right|\leqslant\gamma\frac{L}{K}\left|m_{1}\right|}\frac{c}{\left|m_{1}\right|^{3}}
\leqslant c\frac{K^{2}}{L}\sum^{+\infty}_{\left|m_{1}\right|=1}\frac{1}{\left|m_{1}\right|^{3}}\left(1+\gamma\frac{L}{K}\left|m_{1}\right|\right)\\
\leqslant & c\left(\frac{K^{2}}{L}+K\right)=cK\left(\frac{K}{L}+1\right)\leqslant cK.
\end{align*}
To deal with the term $S_{2}$ we consider the sets
\begin{align*}
\Gamma_{1}= & \left\{ m\in\mathbb{Z}^{2}:\left|Km_{2}\right|>\gamma\left|Lm_{1}\right|\text{ and }\frac{1}{2}L<\left| \left(Lm_{1},Km_{2}\right)\cdot e_{\varphi} \right|\right\} ,\\
\Gamma_{2}= & \left\{ m\in\mathbb{Z}^{2}:\left|Km_{2}\right|>\gamma\left|Lm_{1}\right|\text{ and }1\leqslant\left| \left(Lm_{1},Km_{2}\right)\cdot e_{\varphi}\right|\leqslant\frac{1}{2}L\right\} ,\\
\Gamma_{3}= & \left\{ m\in\mathbb{Z}^{2}\setminus\left\{ 0\right\} :\left|Km_{2}\right|>\gamma\left|Lm_{1}\right|\text{ and }\left|\left(Lm_{1},Km_{2}\right)\cdot e_{\varphi} \right|<1\right\} ,
\end{align*}
and we split the sum in $S_{2}$ as follows
\begin{align*}
S_{2} \leqslant & cL^{2}K^{2}\left(\sum_{m\in\Gamma_{1}}+\sum_{m\in\Gamma_{2}}+\sum_{m\in\Gamma_{3}}\right)\frac{1}{\left|Km_{2}\right|^{2}\left(1+\left|\left(Lm_{1},Km_{2}\right)\cdot e_{\varphi} \right|\right)^{2}}\\
= & S_{21}+S_{22}+S_{23}.
\end{align*}
For the first term we have
\begin{align*}
S_{21}\leqslant & c\sum_{m\in\Gamma_{1}}\frac{1}{\left|m_{2}\right|^{2}\left| m_{1}+\varphi \frac{K}{L}m_{2} \right|^{2}}
\leqslant c\sum_{|m_{2}|\geqslant1}\sum^{+\infty}_{q=0}\frac{1}{\left|m_{2}\right|^{2}2^{2q}}\sum_{m_{1}:\left|m_{1}+\varphi\frac{K}{L}m_{2}\right|\approx2^{q}}1\\
\leqslant & c\sum_{|m_{2}|\geqslant1}\sum^{+\infty}_{q=0}\frac{1}{\left|m_{2}\right|^{2}2^{q}}\leqslant c.
\end{align*}
For the third term we have
\begin{align*}
S_{23}\leqslant & cL^{2}\sum_{m\in\Gamma_{3}}\frac{1}{\left|m_{2}\right|^{2}}
\leqslant cL^{2}\sum_{\stackrel{\left|Km_{2}\right|>\gamma\left|Lm_{1}\right|}{\left|m_{1}+\frac{K}{L}\varphi m_{2}\right|\leqslant L^{-1},m_{2}\neq0}}\frac{1}{\left|m_{2}\right|^{2}}.
\end{align*}
Observe that for every $m_{2}$ there exists at most one integer $m_{1}$
such that $\left|m_{1}+\frac{K}{L}\varphi m_{2}\right|\leqslant L^{-1}$.
Hence
\[
S_{23}\leqslant L^{2}\sum_{1\leqslant\left|m_{2}\right|<+\infty,\left\Vert \varphi\frac{K}{L}m_{2}\right\Vert \leqslant L^{-1},}\frac{1}{\left|m_{2}\right|^{2}}.
\]
By \eqref{eq:irr_alg} and the fact that $L/K\in\mathbb{Z}$ we have
\[
\frac{1}{L} \geqslant \left|m_{1}+\varphi\frac{K}{L}m_{2}\right| \geqslant
\left\Vert \varphi\frac{K}{L}m_{2}\right\Vert \geqslant\frac{K}{L}\left\Vert m_{2}\varphi\right\Vert \geqslant\frac{K}{L}\frac{H}{\left|m_{2}\right|},
\]
so that $KH<\left|m_{2}\right|$. Therefore 
\begin{align*}
S_{23}\leqslant & cL^{2}\sum_{KH\leqslant\left|m_{2}\right|,\left\Vert \varphi\frac{K}{L}m_{2}\right\Vert \leqslant L^{-1},}\frac{1}{\left|m_{2}\right|^{2}}
\leqslant cL^{2}\sum^{+\infty}_{q=\log_2(H K)}\sum_{\substack{2^{q-1}\leqslant\left|m_{2}\right|<2^{q} \\ \left\Vert \varphi\frac{K}{L}m_{2}\right\Vert \leqslant L^{-1}}} 2^{-2q}.
\end{align*}
We now observe that, for every $1\leqslant s\leqslant\frac{2^{q}}{HK}$,
each of the intervals
\[
\left[\frac{sH}{2^{q}}\frac{K}{L},\frac{(s+1)H}{2^{q}}\frac{K}{L}\right)
\]
contains at most two numbers of the form $\left\Vert \varphi\frac{K}{L}m_{2}\right\Vert $.
Indeed, arguing by contradiction, suppose that that one of these intervals contains three or more such numbers. Without loss of generality, we may assume that two of them are of the form  $\left\{ \varphi\frac{K}{L}k_{1}\right\} $
and $\left\{ \varphi\frac{K}{L}k_{2}\right\} $ with $2^{q-1}\leqslant k_{1}<k_{2}<2^{q}$.
Then
\[
\frac{H}{2^{q}}\frac{K}{L}>\left|\left\{ \varphi\frac{K}{L}k_{1}\right\} -\left\{ \varphi\frac{K}{L}k_{2}\right\} \right|=\left\Vert \varphi\frac{K}{L}\left(k_{1}-k_{2}\right)\right\Vert >\frac{K}{L}\frac{H}{k_{2}-k_{1}}>\frac{K}{L}\frac{H}{2^{q}}
\]
which is a contradiction. Therefore
\[
S_{23}\leqslant cL^{2}\sum^{+\infty}_{q=\log_2 (HK)}2^{-2q}\sum^{\frac{2^{q}}{HK}}_{s=1}1\leqslant c\frac{L^{2}}{HK}\sum^{+\infty}_{q=\log_2(HK)}2^{-q}=c\frac{L^{2}}{K^{2}}.
\]
We now consider $S_{22}$. As for $S_{23}$ we note that for every $m_2$ there exists at most one integer $m_1$ such that $\left|m_{1}+\frac{K}{L}\varphi m_{2}\right|\leqslant 1/2$. 
Also $\left|m_{1}+\frac{K}{L}\varphi m_{2}\right|=\left\Vert \frac{K}{L}\varphi m_{2}\right\Vert$. 
Hence $$S_{22}	\leqslant c\sum_{\stackrel{\left|Km_{2}\right|>\gamma\left|Lm_{1}\right|}{\left\Vert \varphi\frac{K}{L}m_{2}\right\Vert \geqslant L^{-1}}}\frac{c}{\left|m_{2}\right|^{2}\left\Vert \varphi\frac{K}{L}m_{2}\right\Vert ^{2}}\leqslant c\sum^{+\infty}_{q=0}\sum_{\substack{2^{q-1}\leqslant m_{2}<2^{q}\\
\left\Vert \varphi\frac{K}{L}m_{2}\right\Vert \geqslant L^{-1}
}
}\frac{c}{\left|m_{2}\right|^{2}\left\Vert \varphi\frac{K}{L}m_{2}\right\Vert ^{2}}.
$$
Similarly to $S_{23}$, if $2^{q-1}\leqslant m_{2}<2^{q}$, then every interval $\left[\frac{sH}{2^{q}}\frac{K}{L},\frac{(s+1)H}{2^{q}}\frac{K}{L}\right)$, $1\leqslant s<\frac{L2^{q}}{HK}$, contains at most two numbers the form $\left\Vert \varphi\frac{K}{L}m_{2}\right\Vert$. Since we require $\left\Vert \varphi\frac{K}{L}m_{2}\right\Vert \geqslant L^{-1}$ we must distinguish the cases $2^{q}\leqslant HK$ and $2^{q}>HK$.
Therefore
\begin{align*}
    S_{22}\leqslant	&c\sum_{2^{q}\leqslant HK}2^{-2q}\sum^{\frac{2^{q}L}{HK}}_{s=1}\frac{1}{\left(\frac{sH}{2^{q}}\frac{K}{L}\right)^{2}}+c\sum_{2^{q}> HK}2^{-2q}\sum^{\frac{2^{q}L}{HK}}_{s=\frac{2^{q}}{HK}}\frac{1}{\left(\frac{sH}{2^{q}}\frac{K}{L}\right)^{2}} \\
\leqslant &	c\frac{L^{2}}{K^{2}}\sum_{2^{q}\leqslant HK}\sum^{\frac{2^{q}L}{HK}}_{s=1}\frac{1}{s^{2}}+c\frac{L^{2}}{K^{2}}\sum_{2^{q}> HK}\sum^{\frac{2^{q}L}{HK}}_{s=\frac{2^{q}}{HK}}\frac{1}{s^{2}} \\
\leqslant &	c\frac{L^{2}}{K^{2}}\sum_{2^{q}\leqslant HK}1+c\frac{L^{2}}{K^{2}}\sum_{2^{q}> HK}\frac{HK}{2^{q}}\leqslant c\frac{L^{2}}{K^{2}}\log K.
\end{align*}
It follows that
\[
\int_{\mathbb{T}^2} \left| D_{N_j}(x)\right|^2dx\leqslant c\left(K+\frac{L^{2}}{K^{2}}\log K\right).
\]
For every positive integer $j$ let $L=j^3$ and $K=j^2$, so that we have $N_j=LK=j^5$ points. Observe that with this choice $\frac{L}{K}=N_j^{1/5}=j$ is an integer and $K=\frac{L^{2}}{K^{2}}$, so that
\[
\int_{\mathbb{T}^2} \left| D_{N_j}(x)\right|^2dx\leqslant c N_j^{\frac{2}{5}}\log N_j.
\]
\end{proof}

\section{Auxiliary results on oscillatory integrals}\label{sec_aux}
This section gathers a few standard estimates for oscillatory integrals that will be used to study the asymptotic behavior of $\widehat{\chi_{\mathcal B}}$. An important role will be played by the Fresnel integrals, which naturally arise after reducing the phase to a quadratic normal form in a neighborhood of a stationary point.

\subsection{Fresnel integrals}
We use the following normalization of the Fresnel integrals:
\[
S(x)=\int_{0}^{x}\sin\!\left(t^{2}\right)dt,
\qquad
C(x)=\int_{0}^{x}\cos\!\left(t^{2}\right)dt.
\]
It will also be convenient to introduce the complex Fresnel tail
\[
\mathfrak{F}(x)=\int_{x}^{+\infty} e^{it^{2}}dt.
\]
In the next lemma, we summarize some classical properties of these functions. 

\begin{lem}
\label{lem:Fresnel}
For every $x\in\left[0,+\infty\right)$ we have
\[
0\leqslant S\left(x\right)\leqslant S\left(\sqrt{\pi}\right)<1
\]
and
\[
0\leqslant C\left(x\right)\leqslant C\left(\sqrt{\frac{\pi}{2}}\right)<1.
\]
Moreover
\[
\lim_{x\to+\infty}S\left(x\right)=\lim_{x\to+\infty}C\left(x\right)=\frac{1}{2}\sqrt{\frac{\pi}{2}}
\]
and in particular
\begin{align}
\label{eq:F(0)}
\mathfrak{F}(0)=\frac{\sqrt{\pi}}{2}e^{i\pi/4}.
\end{align}

Finally, for every $x>0$ we have
\begin{align}
\mathfrak{F}\left(x\right)=-e^{ix^{2}}\frac{1}{2ix}+O\left(\frac{1}{x^{3}}\right).
\label{eq:exp F}
\end{align}
\end{lem}

\begin{proof}
We only prove the asymptotic expansion of $\mathfrak F$ since the other properties are classical. We refer the reader to  \cite{Lebedev} for the details. Let $x>0$, integrating by parts twice gives
\begin{align*}
\mathfrak{F}\left(x\right)
&= -e^{ix^{2}}\frac{1}{2ix}-e^{ix^{2}}\frac{1}{\left(2i\right)^{2}x^{3}}+\int_{x}^{+\infty}e^{iv^{2}}\frac{3}{\left(2i\right)^{2}v^{4}}dv
= -e^{ix^{2}}\frac{1}{2ix}+O\left(\frac{1}{x^{3}}\right).
\end{align*}
\end{proof}

\subsection{Oscillatory integrals}
The asymptotic analysis of oscillatory integrals of the form
\[
\int_a^b e^{i\lambda f(u)}\,g(u)\,du
\]
as $\lambda\to+\infty$ is classical and is governed by the stationary points of the phase $f$.
In what follows we require uniform asymptotic expansions that are valid when one endpoint
of integration approaches a stationary point of $f$.
We do not claim any particular novelty here: rather, we revisit classical arguments (see e.g. \cite{Erdelyi})
 with this uniformity issue in mind.
We begin with the following elementary estimate.

\begin{lem}
\label{lem:Fresnel-1}Let $u_{1}\leqslant0$ and $u_{2}>\eta>0$. Then, for $\lambda$ large enough,
\begin{equation}
\int_{u_{1}}^{u_{2}}e^{i\lambda u^{2}}du=\lambda^{-1/2}\left[2\mathfrak{F}(0) - \mathfrak{F} ({-\lambda^{1/2}u_{1}})\right]+O\left(\frac{1}{\lambda}\right)\label{eq:stima I(lambda)},
\end{equation}
where the implicit constant in O-term  only depends on $\eta$.
\end{lem}

\begin{proof}
Let us write
\begin{align*}
\int_{u_{1}}^{u_{2}}e^{i\lambda u^{2}}du&= \int_{u_{1}}^{0}e^{i\lambda u^{2}}du+\int_{0}^{u_{2}}e^{i\lambda u^{2}}du=\int_{0}^{-u_{1}}e^{i\lambda u^{2}}du+\int_{0}^{u_{2}}e^{i\lambda u^{2}}du\\
&= 2\int_{0}^{+\infty}e^{i\lambda u^{2}}du-\int_{-u_{1}}^{+\infty}e^{i\lambda u^{2}}du-\int_{u_{2}}^{+\infty}e^{i\lambda u^{2}}du\\
&= \lambda^{-1/2}\left[2\mathfrak{F}(0)- 
\mathfrak{F}(-\lambda^{1/2}u_{1})\right]-\lambda^{-1/2}\mathfrak{F}({\lambda^{1/2}u_{2}}).
\end{align*}
For sufficiently large $\lambda$ by  \eqref{eq:exp F} we have $\lambda^{-1/2}\mathfrak{F}(\lambda^{1/2}u_{2})=O((\lambda u_2)^{-1})$ and the lemma follows.
\end{proof}
The previous lemma treats the model phase $u^2$ when one of the endpoints approaches the stationary point $u=0$. We now extend this estimate to a general one-dimensional phase using the standard change of variables near the stationary point.

\begin{prop}
\label{prop:Fase staz}Let $\tau\in\left[a,b\right]$, let $g:[a,b]\to\mathbb{R}$ be a smooth function and
$h_{\tau}:[a,b]\to\mathbb{R}$ be a family of smooth
functions. Assume that 
$
h'_{\tau}\left(\tau\right)=0$
and that there exist positive constants $\eta_{1}$ and $\eta_{2}$,
independent of $\tau$, such that for every $x\in\left[a,b\right]$
and every $\tau$ we have
\begin{equation}
h_{\tau}''\left(x\right)\geqslant\eta_{1}\label{eq:h'' from below}
\end{equation}
and
\begin{equation}
\sum_{j=0}^{4}\left|h_{\tau}^{\left(j\right)}\left(x\right)\right|\leqslant\eta_{2}.\label{eq:derivate h da sopra}
\end{equation}
Then
\[
\int_{a}^{b}e^{i\lambda h_{\tau}\left(x\right)}g\left(x\right)dx=e^{i\lambda h_\tau\left(\tau\right)}\sqrt{\frac{2}{h_\tau''\left(\tau\right)}}g\left(\tau\right)\lambda^{-1/2}\Xi\left(\lambda,\tau\right)+O\left(\lambda^{-1}\right)
\]
where $\Xi\left(\lambda,\tau\right)=\left|\Xi\left(\lambda,\tau\right)\right|e^{i\vartheta\left(\lambda,\tau\right)}$
satisfies
\[
\frac{\sqrt{\pi}}{2}\leqslant\left|\Xi\left(\lambda,\tau\right)\right|\leqslant\left(\frac{\sqrt{\pi}}{2}+\sqrt{2}\right).
\]and
\[
\frac{\pi}{4}-\delta\leqslant\vartheta\left(\lambda,\tau\right)\leqslant\frac{\pi}{4}+\delta
\]
for a suitable $\delta<\frac{\pi}{4}$. Moreover the implicit constant
in the $O\left(\lambda^{-1}\right)$-term depends on $\eta_1$, $\eta_2$ and a finite number of derivatives of $g$ but is independent of
$\tau$.
\end{prop}

\begin{proof}
Since $h_{\tau}''\left(x\right)>0$ and $h'_{\tau}\left(\tau\right)=0$,
the function $h_{\tau}$ attains a minimum at $x=\tau$. It follows that
\begin{equation}
u\left(x\right)=\sqrt{h_{\tau}\left(x\right)-h_{\tau}\left(\tau\right)}\operatorname{sgn}(x-\tau)
\label{eq:u(x)}
\end{equation}
is strictly increasing, hence it is invertible and can be used as
a change of variables. By Taylor's formula with integral
remainder we can write
\[
h_{\tau}\left(x\right)=h_{\tau}\left(\tau\right)+\left(x-\tau\right)^{2}\int_{0}^{1}h_{\tau}''\left(v\tau+\left(1-v\right)x\right)vdv.
\]
Hence
\begin{align*}
u\left(x\right)&= \left(x-\tau\right)\sqrt{\int_{0}^{1}h_{\tau}''\left(v\tau+\left(1-v\right)x\right)vdv}
\end{align*}
which is a smooth function since $h''_{\tau}$ is strictly positive.
Moreover, by \eqref{eq:u(x)}, we have
\begin{equation}
u'\left(x\right)= \frac{h_\tau'\left(x\right)}{2\sqrt{h_\tau\left(x\right)-h_\tau\left(\tau\right)}}\operatorname{sgn}(x-\tau)= \frac{\int_{0}^{1}h_\tau''\left(v\tau+\left(1-v\right)x\right)dv}{2\sqrt{\int_{0}^{1}h_\tau''\left(v\tau+\left(1-v\right)x\right)vdv}}.\label{eq:u'(x)}
\end{equation}
In what follows we shall need uniform bounds for the derivatives for $x^{\left(k\right)}\left(u\right)$
with $k$ up to order $3$. Since $x\left(u\right)$ is the inverse
function of $u\left(x\right)$,  this bound follows from a bound from
below for $u'\left(x\right)$ and a bound from above for $\left|u^{\left(k\right)}\left(x\right)\right|$.
The bound from below follows from (\ref{eq:h'' from below}) and the
bound from above from (\ref{eq:derivate h da sopra}). From (\ref{eq:u'(x)})
we also obtain
\[
u'\left(\tau\right)=\sqrt{\frac{h_\tau''\left(\tau\right)}{2}}
\]
and therefore
\[
x'\left(0\right)=\sqrt{\frac{2}{h_\tau''\left(\tau\right)}}.
\]
Hence
\begin{align*}
\int_{a}^{b}e^{i\lambda h_{\tau}\left(x\right)}g\left(x\right)dx &=  e^{i\lambda h_{\tau}\left(\tau\right)}\int_{a}^{b}g\left(x\right)e^{i\lambda\left[h_{\tau}\left(x\right)-h_{\tau}\left(\tau\right)\right]}dx=  e^{i\lambda h_{\tau}\left(\tau\right)}\int_{u_{1}\left(\tau\right)}^{u_{2}\left(\tau\right)}g\left(x\left(u\right)\right)e^{i\lambda u^{2}}x'\left(u\right)du,
\end{align*}
where $u_{1}\left(\tau\right)=-\sqrt{h_{\tau}\left(a\right)-h_{\tau}\left(\tau\right)}$
and $u_{2}\left(\tau\right)=\sqrt{h_{\tau}\left(b\right)-h_{\tau}\left(\tau\right)}$.
Without loss of generality, we now assume $\tau\in\left[a,\frac{a+b}{2}\right]$, the case $\tau\in\left[\frac{a+b}{2},b\right]$ being similar. Let
$0<\varepsilon<\frac{b-a}{2}$ and observe that 
\begin{align*}
h_{\tau}\left(b\right)-h_{\tau}\left(\tau\right)\geqslant h_{\tau}\left(b\right)-h_{\tau}\left(b-\varepsilon\right)\geqslant & \int_{b-\varepsilon}^{b}h'_{\tau}\left(s\right)ds\geqslant\varepsilon h'_{\tau}\left(b-\varepsilon\right)\\
= &\,\varepsilon\int_{\tau}^{b-\varepsilon}h_{\tau}''\left(s\right)ds\geqslant\varepsilon\eta_{1}\left(\frac{b-a}{2}-\varepsilon\right)
\end{align*}
so that $u_{2}\left(\tau\right)$ is uniformly bounded away from $0$
(in the case $\tau\in\left[\frac{a+b}{2},b\right]$ it will be $u_{1}\left(\tau\right)$
bounded away from $0$). Then
\begin{align*}
& \int_{a}^{b}e^{i\lambda h_{\tau}\left(x\right)}g\left(x\right)dx = \,e^{i\lambda h_{\tau}\left(\tau\right)}g\left(\tau\right)\sqrt{\frac{2}{h''\left(\tau\right)}}\int_{u_{1}\left(\tau\right)}^{u_{2}\left(\tau\right)}e^{i\lambda u^{2}}du+e^{i\lambda h_{\tau}\left(\tau\right)}\int_{u_{1}\left(\tau\right)}^{u_{2}\left(\tau\right)}g_{1}\left(u\right)e^{i\lambda u^{2}}du,
\end{align*}
where
\[
g_{1}\left(u\right)= g\left(x\left(u\right)\right)x'\left(u\right)-g\left(\tau\right)\sqrt{\frac{2}{h_\tau''\left(\tau\right)}}.
\]
Since $g_{1}\left(0\right)=0$, we have
\begin{align*}
g_{1}\left(u\right)&= \int_{0}^{u}\left[g'\left(x\left(t\right)\right)\left[x'\left(t\right)\right]^{2}+g\left(x\left(t\right)\right)x''\left(t\right)\right]dt\\
&= \, u\int_{0}^{1}\left[g'\left(x\left(us\right)\right)\left[x'\left(us\right)\right]^{2}+g\left(x\left(us\right)\right)x''\left(us\right)\right]ds
\end{align*}
so that $g_{1}\left(u\right)=ug_{2}\left(u\right)$.
Since $g_{2}$ and $g_{2}'$ can be bounded uniformly in $\tau$ we
have
\begin{align*}
 & \int_{u_{1}\left(\tau\right)}^{u_{2}\left(\tau\right)}g_{1}\left(u\right)e^{i\lambda u^{2}}du= \int_{u_{1}\left(\tau\right)}^{u_{2}\left(\tau\right)}g_{2}\left(u\right)ue^{i\lambda u^{2}}du \\
 = &\left[g_{2}\left(u\right)\frac{e^{i\lambda u^{2}}}{2\lambda i}\right]_{u_{1}\left(\tau\right)}^{u_{2}\left(\tau\right)}-\frac{1}{2\lambda i}\int_{u_{1}\left(\tau\right)}^{u_{2}\left(\tau\right)}g_{2}'\left(u\right)e^{i\lambda u^{2}}du = O\left(\lambda^{-1}\right).
\end{align*}
By Lemma \ref{lem:Fresnel-1}, we have
\[
\int_{a}^{b}e^{i\lambda h_{\tau}\left(x\right)}g\left(x\right)dx= \lambda^{-1/2}e^{i\lambda h_\tau\left(\tau\right)}g\left(\tau\right)\sqrt{\frac{2}{h_\tau''\left(\tau\right)}}\Xi\left(\lambda,\tau\right)+O\left(\lambda^{-1}\right),
\]
where 
\begin{align*}
\Xi\left(\lambda,\tau\right)&= 2\mathfrak{F}(0)-\mathfrak{F}(-\lambda^{1/2}u_{1}(\tau))=\int_{0}^{+\infty}e^{iv^{2}}dv+\int_{0}^{-\lambda^{1/2}u_{1}\left(\tau\right)}e^{iv^{2}}dv.
\end{align*}
By Lemma \ref{lem:Fresnel}
\[
\sqrt{\frac{\pi}{8}}=\int_{0}^{+\infty}\cos\left(v^{2}\right)dv\leqslant\mathrm{Re}\left(\Xi\left(\lambda,\tau\right)\right)\leqslant\sqrt{\frac{\pi}{8}}+1
\]
and
\[
\sqrt{\frac{\pi}{8}}=\int_{0}^{+\infty}\sin\left(v^{2}\right)dv\leqslant\mathrm{Im}\left(\Xi\left(\lambda,\tau\right)\right)\leqslant\sqrt{\frac{\pi}{8}}+1.
\]
Hence
\[
\frac{\sqrt{\pi}}{2}\leqslant\left|\Xi\left(\lambda,\tau\right)\right|\leqslant\left(\frac{\sqrt{\pi}}{2}+\sqrt{2}\right).
\]
If $\vartheta\left(\lambda,\tau\right)=\arg\Xi\left(\lambda,\tau\right)
$, we also have
\[
\frac{\pi}{4}-\delta\leqslant\vartheta\left(\lambda,\tau\right)\leqslant\frac{\pi}{4}+\delta, 
\]
for a suitable $0<\delta<\pi/4.$

\end{proof}

The following lemma gives an estimate for the quadratic phase when the critical point lies outside the interval of integration. This estimate will later be used to study the transition from the non-stationary regime to the stationary one.

\begin{lem}
\label{lem:Fresnel pto critico esterno}Let $0\leqslant u_{1}\leqslant\eta\leqslant u_{2}\leqslant K$,
let $\psi\in C^{2}\left(\mathbb{R}\right)$. 
Then, as $\lambda\rightarrow+\infty$,
\[
\int_{u_{1}}^{u_{2}}e^{i\lambda u^{2}}\psi\left(u\right)du=\psi\left(0\right)\lambda^{-1/2}\mathfrak{F}\left(\lambda^{1/2}u_{1}\right)+O\left(\frac{1}{\lambda}\right).
\]
In particular, if $\lambda^{1/2}u_{1}$ is large,
\begin{equation}
\int_{u_{1}}^{u_{2}}e^{i\lambda u^{2}}\psi\left(u\right)du=\frac{1}{\lambda u_1}\left(e^{i\lambda u_{1}^{2}}\frac{i}{2}\psi\left(0\right)+O\left(\frac{1}{\left(\lambda^{1/2}u_{1}\right)^{2}}\right)+O\left(u_1\right)\right).\label{eq:10}
\end{equation}
Moreover the implicit constants
in the O--terms only depend on $\eta$, $K$, and the derivatives of $\psi$.
\end{lem}

\begin{proof}
Let 
\[
\psi_{1}(u)=\frac{\psi\left(u\right)-\psi\left(0\right)}{2iu}=\frac{1}{2i}\int_0^1\psi'(tu)dt.
\]
Since $\psi\in C^2(\mathbb{R})$ it follows that $\psi_1\in C^1(\mathbb R)$. Hence
\begin{align*}
\int_{u_{1}}^{u_{2}}e^{i\lambda u^{2}}\psi\left(u\right)du&= \psi\left(0\right)\int_{u_{1}}^{u_{2}}e^{i\lambda u^{2}}du+\frac{1}{\lambda}\int_{u_{1}}^{u_{2}}2ui\lambda e^{i\lambda u^{2}}\psi_{1}\left(u\right)du\\
&= \psi\left(0\right)\lambda^{-1/2}\int_{\lambda^{1/2}u_{1}}^{\lambda^{1/2}u_{2}}e^{iv^{2}}dv+\frac{1}{\lambda}\left[e^{i\lambda u^{2}}\psi_{1}\left(u\right)\right]_{u_{1}}^{u_{2}}-\frac{1}{\lambda}\int_{u_{1}}^{u_{2}}e^{i\lambda u^{2}}\psi_{1}'\left(u\right)du\\
&= \psi\left(0\right)\lambda^{-1/2}\left[\mathfrak{F}\left(\lambda^{1/2}u_{1}\right)-\mathfrak{F}\left(\lambda^{1/2}u_{2}\right)\right]+O\left(\frac{1}{\lambda}\right).
\end{align*}
Since $u_{2}\geqslant\eta$, by \eqref{eq:exp F} we have
\[
\mathfrak{F}\left(\lambda^{1/2}u_{2}\right)=O\left(\frac{1}{\lambda^{1/2}}\right).
\]
Hence
\[
\int_{u_{1}}^{u_{2}}e^{i\lambda u^{2}}\psi\left(u\right)du=\psi\left(0\right)\lambda^{-1/2}\mathfrak{F}\left(\lambda^{1/2}u_{1}\right)+O\left(\frac{1}{\lambda}\right).
\]
Using again \eqref{eq:exp F} we obtain 
\begin{align*}
\int_{u_{1}}^{u_{2}}e^{i\lambda u^{2}}\psi\left(u\right)du&=\lambda^{-1/2}\left(e^{i\lambda u_{1}^{2}}\frac{i}{2\lambda^{1/2}u_{1}}\psi\left(0\right)+O\left(\frac{1}{\left(\lambda^{1/2}u_{1}\right)^{3}}\right)+O\left(\frac{1}{\lambda^{1/2}}\right)\right)\\
&=\frac{1}{\lambda u_1}\left(e^{i\lambda u_{1}^{2}}\frac{i}{2}\psi\left(0\right)+O\left(\frac{1}{\left(\lambda^{1/2}u_{1}\right)^{2}}\right)+O\left(u_1\right)\right)
\end{align*}
which gives \eqref{eq:10}.
\end{proof}

We now extend the preceding estimate for the quadratic phase to a general one-dimensional phase whose critical point lies outside the interval of integration and may approach one of its endpoints.

\begin{prop}\label{prop:tau fuori}
Let $a<b$, let $\varepsilon>0$ and let $\tau\in\left[a-\varepsilon,a\right]$.
Let $g:\left[a-\varepsilon,b\right]\to\mathbb{R}$ be a smooth function
and let $h_{\tau}:\left[a-\varepsilon,b\right]\to\mathbb{R}$ be a
family of smooth functions. Assume that
$
h_{\tau}'\left(\tau\right)=0
$
and that there exist positive constants $\eta_{1},\eta_{2}$ independent
if $\tau$ such that for every $x\in\left[a-\varepsilon,b\right]$
and every $\tau\in\left[a-\varepsilon,a\right]$ we have
\[
h_{\tau}''\left(x\right)\geqslant\eta_{1}
\]
and
\[
\sum_{j=0}^{4}\left|h_{\tau}^{\left(j\right)}\left(x\right)\right|\leqslant\eta_{2}.
\]
Then
\[
\int_{a}^{b}e^{i\lambda h_{\tau}\left(x\right)}g\left(x\right)dx=e^{i\lambda h_{\tau}\left(\tau\right)}g\left(\tau\right)\sqrt{\frac{2}{h''\left(\tau\right)}}\lambda^{-1/2}\mathfrak{F}\left(\lambda^{1/2}\sqrt{h(a)-h(\tau)}\right)+O\left(\frac{1}{\lambda}\right),
\]
and the error term $O(\frac 1\lambda)$ is uniform in $\tau\in[a-\varepsilon,a]$.
\end{prop}

\begin{proof}
Since $h_{\tau}'$ is strictly increasing, it follows that $h_{\tau}'\left(x\right)=0$
if and only if $x=\tau$. By the Taylor's formula with integral remainder
\[
h\left(x\right)-h\left(\tau\right)=\left(x-\tau\right)^{2}\int_{0}^{1}h''\left(\tau+s\left(x-\tau\right)\right)\left(1-s\right)ds
\]
so that, for every $x\in\left[a,b\right]$, we have
\begin{equation}
\sqrt{h\left(x\right)-h\left(\tau\right)}=\left(x-\tau\right)\sqrt{\int_{0}^{1}h''\left(\tau+s\left(x-\tau\right)\right)\left(1-s\right)ds}.\label{eq:Taylor integral}
\end{equation}
Since the integral is strictly positive it follows that $u=\sqrt{h\left(x\right)-h\left(\tau\right)}$
is a smooth change of variables. Hence
\[
\int_{a}^{b}e^{i\lambda h_{\tau}\left(x\right)}g\left(x\right)dx=e^{i\lambda h_{\tau}\left(\tau\right)}\int_{u_{1}}^{u_{2}}e^{i\lambda u^{2}}\psi\left(u\right)du
\]
where $
u_{1}=\sqrt{h\left(a\right)-h\left(\tau\right)}$,  $u_{2}=\sqrt{h\left(b\right)-h\left(\tau\right)},
$
and $
\psi\left(u\right)=g\left(x\left(u\right)\right)x'\left(u\right)$. 
Since
\[
u'\left(x\right)=\frac{h'\left(x\right)}{2\sqrt{h\left(x\right)-h\left(\tau\right)}}=\frac{\int_{0}^{1}h''\left(\tau+s\left(x-\tau\right)\right)ds}{2\sqrt{\int_{0}^{1}h''\left(\tau+s\left(x-\tau\right)\right)\left(1-s\right)ds}},
\]
arguing as in the proof of Proposition \ref{prop:Fase staz} we see that
$\psi\in C^{2}$. Observe now that by (\ref{eq:Taylor integral})
we have $
u_{1}\leqslant c\left(a-\tau\right)\leqslant c\varepsilon$. 
and $u_{2}\geqslant c\left(b-a\right)$. By Lemma \ref{lem:Fresnel pto critico esterno}, we have
\[
\int_{u_{1}}^{u_{2}}e^{i\lambda u^{2}}\psi\left(u\right)du=\psi\left(0\right)\lambda^{-1/2}\mathfrak{F}\left(\lambda^{1/2}u_{1}\right)+O\left(\frac{1}{\lambda}\right).
\]
Since
\[
x'\left(0\right)=\frac{1}{u'\left(\tau\right)}=\sqrt{\frac{2}{h''\left(\tau\right)}},
\]
the proposition follows. 
\end{proof}

\begin{rem}
An analogous statement holds when the stationary point lies to the right of the integration interval. More precisely, if $\tau \in [b, b+\varepsilon]$, the same asymptotic expansion remains valid, provided the assumptions are adapted accordingly.  The resulting formula is obtained by replacing $h_\tau(a)-h_\tau(\tau)$ with $h_\tau(b)-h_\tau(\tau)$ in the argument of $\mathfrak{F}$.
\end{rem}

\subsection{A Bessel representation of $\widehat{\chi_{\mathcal B}}$}

Since the barrel $\mathcal{B}$ is invariant under rotations about the $x_d$ axis, we have
\[
\widehat{\chi_{\mathcal{B}}}\left(\rho\Theta\right)=\widehat{\chi_{\mathcal{B}}}\left(\rho\Theta',\rho\Theta_{d}\right)=\widehat{\chi_{\mathcal{B}}}\left(\rho\left|\Theta'\right|,0,\ldots,0,\rho\Theta_{d}\right).
\]
Thus, in deriving the formula below, we may assume without loss of generality that $\Theta=\left(\Theta_{1},0,\ldots,0,\Theta_{d}\right)$.
With this convention,
\begin{align*}
\widehat{\chi_{\mathcal{B}}}\left(\rho\Theta\right)&= \int_{-\beta}^{\beta}\int_{\left|x\right|\leqslant R\left(z\right)}e^{-2\pi i\rho\left(\Theta_{1}x_{1}+\Theta_{d}z\right)}dxdz,
\end{align*}
where
$$
R\left(z\right)=-\alpha+\sqrt{1-z^{2}}.$$

Due to the symmetry of $\mathcal{B}$, we can also express its Fourier transform in terms of Bessel functions.

\begin{lem}
\label{lem:Fourier int Bessel} Assume $\Theta_1\neq0$. Then
\begin{align*}
\widehat{\chi_{\mathcal{B}}}\left(\rho\Theta\right)&= \frac{1}{\left|\rho\Theta_{1}\right|^{\frac{d-1}{2}}}\int_{-\beta}^{\beta}e^{-2\pi i\rho\Theta_{d}z}J_{\frac{d-1}{2}}\left(2\pi\rho\left|\Theta_{1}\right|R\left(z\right)\right)R\left(z\right)^{\frac{d-1}{2}}dz.
\end{align*}
\end{lem}

\begin{proof}
Let $B'=\left\{ v\in\mathbb{R}^{d-1}:\left|v\right|\leqslant1\right\} $.
Since $\widehat{\chi_{B'}}\left(\xi\right)=J_{\frac{d-1}{2}}\left(2\pi\left|\xi\right|\right)\left|\xi\right|^{\frac{1-d}{2}}$
(see e.g. \cite[Theorem 4.15]{SW}), a straightforward computation gives
\begin{align}
\widehat{\chi_{\mathcal{B}}}\left(\rho\Theta\right)&= \int_{-\beta}^{\beta}e^{-2\pi i\rho\Theta_{d}z}\int_{B'}e^{-2\pi i\rho\left(\Theta_{1}R\left(z\right)v_{1}\right)}dvR\left(z\right)^{d-1}dz\label{eq:chi_omega}\\
&= \int_{-\beta}^{\beta}e^{-2\pi i\rho\Theta_{d}z}\widehat{\chi_{B'}}\left(\Theta_{1}\rho R\left(z\right)e_{1}\right)R\left(z\right)^{d-1}dz\nonumber \\
&= \frac{1}{\left|\rho\Theta_{1}\right|^{\frac{d-1}{2}}}\int_{-\beta}^{\beta}e^{-2\pi i\rho\Theta_{d}z}J_{\frac{d-1}{2}}\left(2\pi\rho\left|\Theta_{1}\right|R\left(z\right)\right)R\left(z\right)^{\frac{d-1}{2}}dz.\nonumber 
\end{align}
\end{proof}
\begin{lem}
\label{lem:5}Let
\begin{align*}
I\left(\rho,\Theta_{1}\right)&= \int_{-\beta}^{\beta}e^{-2\pi i\rho\Theta_{d}x}J_{\frac{d-1}{2}}\left(2\pi\rho\left|\Theta_{1}\right|R\left(x\right)\right)R\left(x\right)^{\frac{d-1}{2}}dx.
\end{align*}
Then there exist $c_{1}>0$ and $c_{2}\geqslant0$ such that for $\rho|\Theta_{1}|$
large we have
\begin{align*}
I\left(\rho,\Theta_{1}\right) & =\frac{c_{1}}{\left(\rho\left|\Theta_{1}\right|\right)^{1/2}}\mathrm{Re}\left(e^{id\frac{\pi}{4}}\int_{-\beta}^{\beta}e^{-2\pi i\rho\left(\Theta_{d}x+\left|\Theta_{1}\right|R\left(x\right)\right)}R\left(x\right)^{\frac{d-2}{2}}dx\right)\\
 & +\frac{c_{2}}{\left(\rho\left|\Theta_{1}\right|\right)^{3/2}}\mathrm{Im}\left(e^{id\frac{\pi}{4}}\int_{-\beta}^{\beta}e^{-2\pi i\rho\left(\Theta_{d}x+\left|\Theta_{1}\right|R\left(x\right)\right)}R\left(x\right)^{\frac{d-4}{2}}dx\right)\\
 & +\min\left(1,\frac{1}{\rho|\Theta_d|}\right)O\left(\frac{1}{\left(\rho\left|\Theta_{1}\right|\right)^{3/2}}\right).
 \end{align*}
\end{lem}

\begin{proof}
By the asymptotic expansion of Bessel functions for large argument, see e.g. \cite[10.17.3]{DLMF}, we have 
\[
J_{\nu}\left(z\right)=\left(\frac{2}{\pi z}\right)^{1/2}\left(\cos\left(z-\frac{2\nu+1}{4}\pi\right)\left(1+E_{0}\left(z\right)\right)-\sin\left(z-\frac{2\nu+1}{4}\pi\right)\left(\frac{4\nu^{2}-1}{8z}+E_{1}\left(z\right)\right)\right)
\]
with error terms satisfying $E_0(z)=O(z^{-2})$ and $E_1(z)=O(z^{-3})$.
Moreover, by \cite[Theorem 4.2]{Olver} we have $E_0'(z)=O(z^{-3})$ and $E_1'(z)=O(z^{-4})$. Therefore 
\[
J_{\nu}\left(z\right)=\left(\frac{2}{\pi z}\right)^{\frac{1}{2}}\left(\cos\left(z-\frac{2\nu+1}{4}\pi\right)-\frac{4\nu^{2}-1}{8z}\sin\left(z-\frac{2\nu+1}{4}\pi\right)+E\left(z\right)\right)
\]
with $E\left(z\right)=O\left(z^{-2}\right)$ and $E'\left(z\right)=O\left(z^{-2}\right)$.
Hence,
\begin{align*}
I\left(\rho,\Theta_{1}\right)=& \int_{-\beta}^{\beta}e^{-2\pi i\rho\Theta_{d}x}J_{\frac{d-1}{2}}\left(2\pi\rho\left|\Theta_{1}\right|R\left(x\right)\right)R\left(x\right)^{\frac{d-1}{2}}dx\\
= &\frac{1}{\pi}\frac{1}{\left(\rho\left|\Theta_{1}\right|\right)^{1/2}}\int_{-\beta}^{\beta}e^{-2\pi i\rho\Theta_{d}x}\cos\left(2\pi\rho\left|\Theta_{1}\right|R\left(x\right)-d\frac{\pi}{4}\right)R\left(x\right)^{\frac{d-2}{2}}dx\\
 & -\frac{d^{2}-2d}{16\pi^{2}}\frac{1}{\left(\rho\left|\Theta_{1}\right|\right)^{3/2}}\int_{-\beta}^{\beta}e^{-2\pi i\rho\Theta_{d}x}\sin\left(2\pi\rho\left|\Theta_{1}\right|R\left(x\right)-d\frac{\pi}{4}\right)R\left(x\right)^{\frac{d-4}{2}}dx\\
 & +\frac{1}{\pi}\frac{1}{\left(\rho\left|\Theta_{1}\right|\right)^{1/2}}\int_{-\beta}^{\beta}e^{-2\pi i\rho\Theta_{d}x}E\left(2\pi\rho\left|\Theta_{1}\right|R\left(x\right)\right)R\left(x\right)^{\frac{d-2}{2}}dx\\
=& \frac{1}{\pi}\frac{1}{\left(\rho\left|\Theta_{1}\right|\right)^{1/2}}A\left(\rho,\Theta\right)-\frac{d^{2}-2d}{16\pi^{2}}\frac{1}{\left(\rho\left|\Theta_{1}\right|\right)^{3/2}}B\left(\rho,\Theta\right)+\frac{1}{\pi}\frac{1}{\left(\rho\left|\Theta_{1}\right|\right)^{1/2}}C\left(\rho,\Theta\right).
\end{align*}
Using that $R$ is even and that the interval of integration is symmetric, we obtain
\begin{align*}
A\left(\rho,\Theta\right)=& \int_{-\beta}^{\beta}e^{-2\pi i\rho\Theta_{d}x}\cos\left(2\pi\rho\left|\Theta_{1}\right|R\left(x\right)-d\frac{\pi}{4}\right)R\left(x\right)^{\frac{d-2}{2}}dx\\
=&\, \mathrm{Re}\left(e^{id\frac{\pi}{4}}\int_{-\beta}^{\beta}e^{-2\pi i\rho\left(\Theta_{d}x+\left|\Theta_{1}\right|R\left(x\right)\right)}R\left(x\right)^{\frac{d-2}{2}}dx\right),
\end{align*}
and similarly
\begin{align*}
B\left(\rho,\Theta\right)&=- \mathrm{Im}\left(e^{id\frac{\pi}{4}}\int_{-\beta}^{\beta}e^{-2\pi i\rho\left(\Theta_{d}x+\left|\Theta_{1}\right|R\left(x\right)\right)}R\left(x\right)^{\frac{d-4}{2}}dx\right).
\end{align*}
Finally we have
\begin{align*}
C\left(\rho,\Theta\right)= &	\int_{-\beta}^{\beta}\cos\left(2\pi\rho\Theta_{d}x\right)E\left(2\pi\rho\left|\Theta_{1}\right|R\left(x\right)\right)R\left(x\right)^{\frac{d-2}{2}}dx \\
= & 	\left[\frac{\sin\left(2\pi\rho\Theta_{d}x\right)}{2\pi\rho\Theta_{d}}E\left(2\pi\rho\left|\Theta_{1}\right|R\left(x\right)\right)R\left(x\right)^{\frac{d-2}{2}}\right]_{-\beta}^{\beta} \\
	& -\int_{-\beta}^{\beta}\frac{\sin\left(2\pi\rho\Theta_{d}x\right)}{2\pi \rho\Theta_{d}}\left[E'\left(2\pi\rho\left|\Theta_{1}\right|R\left(x\right)\right)2\pi\rho\left|\Theta_{1}\right|R'\left(x\right)R\left(x\right)^{\frac{d-2}{2}}\right.
	\\ 
    & \quad \left. +\frac{d-2}{2}E\left(2\pi\rho\left|\Theta_{1}\right|R\left(x\right)\right)R\left(x\right)^{\frac{d-4}{2}}R'\left(x\right)\right]dx \\
=& \,\frac{\sin\left(2\pi\rho\Theta_{d}\beta\right)}{\pi\rho\Theta_{d}}O\left(\frac{1}{\left(\rho\Theta_{1}\right)^{2}}\right)+O\left(\frac{\sin\left(2\pi\rho\Theta_{d}\beta\right)}{\pi\rho\Theta_{d}}\right)O\left(\frac{1}{\rho\Theta_{1}}\right) \\
=& \,\min\left(1,\frac{1}{\rho\left|\Theta_{d}\right|}\right)O\left(\frac{1}{\rho\left|\Theta_{1}\right|}\right).
\end{align*}
\end{proof}

\section{The Fourier transform of the characteristic function of the barrel}\label{sec_fourier}

For convenience, in what follows we set 
\begin{align*}
h(x)= &\,h_\Theta(x)=-2\pi\left(x\Theta_d+|\Theta'| R(x)\right),\\ 
\beta_*= &\,\mathrm{sgn}(\Theta_d)\beta,\\
u_*= &\,\sqrt{h(\beta_*)-h(\Theta_d)}.
\end{align*} 

We now state the asymptotic estimates for the Fourier transform of
$\chi_{\mathcal B}$. The implicit constants in the $O$-terms below depend on
$\mathcal B$, equivalently on $\alpha$ and $\beta$. We shall not keep track
of this dependence. If a constant depends on an additional parameter, say $\varepsilon$, we shall write $O_\varepsilon$.
 
\begin{thm} Let $0<\varepsilon<1-\beta$. The following estimates hold.
\label{thm:1}
\begin{enumerate}
\item Let $\left|\Theta_{d}\right|\leqslant\beta$. Then, there exist $\delta\in(0,\frac\pi4)$ and a function  $\Xi\left(\rho,\Theta\right)=\left|\Xi\left(\rho,\Theta\right)\right|e^{i\vartheta\left(\rho,\Theta\right)}$ satisfying
\[
\left|\Xi\left(\rho,\Theta\right)\right|\approx 1\]
and
\[
\frac{\pi}{4}-\delta\leqslant\vartheta\left(\rho,\Theta\right)\leqslant\frac{\pi}{4}+\delta
\]
such that
\[
\widehat{\chi_{\mathcal{B}}}\left(\rho\Theta\right)=\frac{1}{\rho^{\frac{d+1}{2}}}\frac{R\left(\Theta_{d}\right)^{\frac{d-2}{2}}}{\left|\Theta'\right|^{\frac{d-2}{2}}}{\left|\Xi\left(\rho,\Theta\right)\right|}\cos\left(\rho h(\Theta_d)+\vartheta\left(\rho,\Theta\right)+d\frac{\pi}{4}\right)+O\left(\frac{1}{\rho^{\frac{d+2}{2}}}\right).
\]
\item Let $\beta<\left|\Theta_{d}\right|\leqslant\beta+\varepsilon$.
Then
\begin{align*}
\widehat{\chi_{\mathcal{B}}}\left(\rho\Theta\right)&= \frac{c}{\rho^{\frac{d+1}{2}}}\frac{R\left(\Theta_{d}\right)^{\frac{d-2}{2}}}{\left|\Theta'\right|^{\frac{d-2}{2}}}\mathrm{Re}\left(e^{id\frac{\pi}{4}}e^{i\rho h\left(\Theta_{d}\right)}\mathfrak{F}\left(\rho^{1/2}u_*\right)\right)+O\left(\frac{1}{\rho^{\frac{d+2}{2}}}\right).
\end{align*}
\item Let $\beta<\left|\Theta_{d}\right|\leqslant\beta+\varepsilon$. Then
\[
\widehat{\chi_{\mathcal{B}}}\left(\rho\Theta\right)=\frac{c}{u_*\rho^{\frac{d+2}{2}}}\left(\frac{R\left(\Theta_{d}\right)^{\frac{d-2}{2}}}{\left|\Theta'\right|^{\frac{d-2}{2}}}\cos\left(\rho h\left(\beta_*\right)+\frac{d-2}{4}\pi\right)+O\left(\frac{1}{\rho\left(\left|\Theta_{d}\right|-\beta\right)^{2}}\right)+O\left(\varepsilon\right)\right).
\]
\item Let $\beta+\varepsilon<\left|\Theta_{d}\right|\leqslant1$. Then
\[
\widehat{\chi_{\mathcal{B}}}\left(\rho\Theta\right)=R\left(\beta\right)^{d-1}\left(\frac{J_{\frac{d-1}{2}}\left(2\pi\rho\left|\Theta'\right|R\left(\beta\right)\right)}{\left|\rho\Theta'\right|^{\frac{d-1}{2}}R\left(\beta\right)^{\frac{d-1}{2}}}\right)\frac{\sin\left(2\pi\rho\Theta_{d}\beta\right)}{\pi\rho\Theta_{d}}+O\left(\frac{1}{\rho^{2}}\right).
\]
In particular there exist $K>0$ and $c>0$ such that if $\left|2\pi\rho R\left(\beta\right)\Theta'\right|<K$
then
\begin{equation}
\left|\widehat{\chi_{\mathcal{B}}}\left(\rho\Theta\right)\right|\geqslant cR\left(\beta\right)^{d-1}\frac{\left|\sin\left(2\pi\rho\Theta_{d}\beta\right)\right|}{\pi\rho|\Theta_{d}|}+O\left(\frac{1}{\rho^{2}}\right).\label{eq:da sotto 4) thm 1}
\end{equation}
\item Let $\beta+\varepsilon<\left|\Theta_{d}\right|\leqslant1$, then for $\rho|\Theta'|$ large
\begin{align}
\widehat{\chi_{\mathcal{B}}}\left(\rho\Theta\right) & =c\frac{R\left(\beta\right)^{\frac{d-2}{2}}}{\rho\left|\rho\Theta'\right|^{\frac{d}{2}}}\left(\frac{1}{\Theta_{d}+\left|\Theta'\right|\frac{\beta}{\sqrt{1-\beta^{2}}}}\cos\left(\rho h(-\beta)+\left(d-2\right)\frac{\pi}{4}\right)\right.\nonumber \\
 & \left.-\frac{1}{\Theta_{d}-\left|\Theta'\right|\frac{\beta}{\sqrt{1-\beta^{2}}}}\cos\left(\rho h\left(\beta\right)+\left(d-2\right)\frac{\pi}{4}\right)\right)\label{eq:due cos-1} +O_\varepsilon\left(\frac{1}{\rho\left|\rho\Theta'\right|^{\frac{d+2}{2}}}\right). 
\end{align}
\end{enumerate}
\end{thm}

We split the proof of the theorem into a series of propositions. In what follows, using again the invariance of $\widehat{\chi_{\mathcal B}}$ under rotations about the $x_d$-axis, we may assume without loss of generality that
\[
\Theta=(\Theta_1,0,\dots,0,\Theta_d),
\]
so that $|\Theta_1|=|\Theta'|$.

\subsection{The region $\left|\Theta_{d}\right|\leqslant\beta$}
\begin{prop}
\label{lem:10}
Let $\left|\Theta_{d}\right|\leqslant\beta$. Then there exist
$\delta\in(0,\frac\pi4)$ and 
$\Xi(\rho,\Theta)=|\Xi(\rho,\Theta)|e^{i\vartheta(\rho,\Theta)}$ satisfying
\begin{equation}\label{eq:xi}
\left|\Xi\left(\rho,\Theta\right)\right|\approx 1
\end{equation}
and
\begin{equation}\label{eq:theta}
\frac{\pi}{4}-\delta\leqslant\vartheta\left(\rho,\Theta\right)\leqslant\frac{\pi}{4}+\delta
\end{equation}
such that
\[
\widehat{\chi_{\mathcal{B}}}\left(\rho\Theta\right)=
\frac{1}{\rho^{\frac{d+1}{2}}}
\frac{R\left(\Theta_{d}\right)^{\frac{d-2}{2}}}{\left|\Theta_{1}\right|^{\frac{d-2}{2}}}
{\left|\Xi\left(\rho,\Theta\right)\right|}
\cos\left(\rho h(\Theta_d)+\vartheta\left(\rho,\Theta\right)+d\frac{\pi}{4}\right)
+
O\left(\frac{1}{\rho^{\frac{d+2}{2}}}\right).
\]
\end{prop}

\begin{proof}
Since $|\Theta_1|\geqslant\sqrt{1-\beta^2}$, by Lemma \ref{lem:Fourier int Bessel} and Lemma \ref{lem:5} we have
\begin{align}
\widehat{\chi_{\mathcal{B}}}\left(\rho\Theta\right) & =\frac{1}{\left|\rho\Theta_{1}\right|^{\frac{d-1}{2}}}\int_{-\beta}^{\beta}e^{-2\pi i\rho\Theta_{d}z}J_{\frac{d-1}{2}}\left(2\pi\rho\left|\Theta_{1}\right|R\left(z\right)\right)R\left(z\right)^{\frac{d-1}{2}}dz\nonumber\\
&= \frac{c_{1}}{\left(\rho\left|\Theta_{1}\right|\right)^{d/2}}\mathrm{Re}\left(e^{id\frac{\pi}{4}}\int_{-\beta}^{\beta}e^{i\rho h(x)}R\left(x\right)^{\frac{d-2}{2}}dx\right)+O\left(\frac{1}{\rho^{\frac{d+2}{2}}}\right)\label{eq:prop13}.
\end{align}
Set
\[
H\left(\rho\right)= \int_{-\beta}^{\beta}e^{i\rho h\left(x\right)}R\left(x\right)^{\frac{d-2}{2}}dx.
\]
The phase $h$ has a stationary point at $x=\Theta_d$. Indeed,

\[
h'(x)=-2\pi\left(\Theta_d-|\Theta_1|\frac{x}{\sqrt{1-x^2}}\right),
\]
and therefore $h'(\Theta_d)=0$. Moreover
\[
h''\left(x\right)=2\pi\left|\Theta_{1}\right|\left(1-x^{2}\right)^{-3/2}\geqslant 2\pi\sqrt{1-\beta^2}.
\]
Hence, by Proposition \ref{prop:Fase staz}, we have
\[
H\left(\rho\right)=e^{i\rho h(\Theta_d)}\left|\Theta_{1}\right|R\left(\Theta_{d}\right)^{\frac{d-2}{2}}\rho^{-1/2}{\Xi\left(\rho,\Theta\right)}+O\left(\rho^{-1}\right),
\]
with $\Xi$ satisfying \eqref{eq:xi} and \eqref{eq:theta}. 
Then
\[
  \mathrm{Re}\left(e^{id\frac{\pi}{4}}H(\rho)\right)=\left|\Theta_{1}\right|R\left(\Theta_{d}\right)^{\frac{d-2}{2}}\rho^{-1/2}\left|\Xi\left(\rho,\Theta\right)\right|\cos\left(\rho h(\Theta_d)+\vartheta\left(\rho,\Theta\right)+d\frac{\pi}{4}\right)+O\left(\rho^{-1}\right).
\]
Substituting this estimate into \eqref{eq:prop13} gives the result.
\end{proof}

\subsection{The region $\beta<\left|\Theta_{d}\right|\leqslant\beta+\varepsilon$.}

We start with the following.
\begin{prop}
\label{prop:Stima beta beta+eps}Let $\beta<\left|\Theta_{d}\right|\leqslant\beta+\varepsilon$.
Then
\begin{align*}
\widehat{\chi_{\mathcal{B}}}\left(\rho\Theta\right)&= \frac{c}{\rho^{\frac{d+1}{2}}}\frac{R\left(\Theta_{d}\right)^{\frac{d-2}{2}}}{ \left|\Theta_{1}\right|^{\frac{d-2}{2}}}\mathrm{Re}\left(e^{id\frac{\pi}{4}}e^{i\rho h\left(\Theta_{d}\right)}\mathfrak{F}\left(\rho^{1/2}u_*\right)\right)+O\left(\frac{1}{\rho^{\frac{d+2}{2}}}\right),
\end{align*}
where
\begin{equation}\label{eq:u ast}
u_*=\sqrt{h(\beta_*)-h(\Theta_d)}.
\end{equation}
In particular $u_*\approx\left(\left|\Theta_{d}\right|-\beta\right)$.
\end{prop}

\begin{proof}
As in
the proof of Proposition \ref{lem:10} we write
\begin{align}
\widehat{\chi_{\mathcal{B}}}\left(\rho\Theta\right) & =\frac{c_{1}}{\left(\rho\left|\Theta_{1}\right|\right)^{d/2}}\mathrm{Re}\left(e^{id\frac{\pi}{4}}\int_{-\beta}^{\beta}e^{i\rho h(x)}R\left(x\right)^{\frac{d-2}{2}}dx\right)+O\left(\frac{1}{\left|\rho\Theta_{1}\right|^{\frac{d+2}{2}}}\right)\label{eq:3.2.1}
\end{align}
and we consider again
\begin{align*}
H\left(\rho\right)&= \int_{-\beta}^{\beta}e^{i\rho h\left(x\right)}R\left(x\right)^{\frac{d-2}{2}}dx.
\end{align*}
Since $h'(x)=0$ if and only if $x=\Theta_d$ and $h''(x)\geqslant 2\pi|\Theta_1|\geqslant 2\pi\sqrt{1-(\beta+\varepsilon)^2}$, by Proposition \ref{prop:tau fuori} (and the subsequent remark) we obtain
\[
H\left(\rho\right)=e^{i\rho h\left(\Theta_{d}\right)}R\left(\Theta_{d}\right)^{\frac{d-2}{2}}\left|\Theta_{1}\right|\pi^{-1/2}\rho^{-1/2}{\mathfrak{F}\left(\rho^{1/2}u_*\right)}+O\left(\frac{1}{\rho}\right)
\]
with $u_*$ as in \eqref{eq:u ast}.
By (\ref{eq:3.2.1}) the proposition follows.
\end{proof}
\begin{prop}
\label{prop:Fresnel arg grande}Let $\beta<\left|\Theta_{d}\right|\leqslant\beta+\varepsilon$, then
\begin{align*}
\widehat{\chi_{\mathcal{B}}}\left(\rho\Theta\right)&= \frac{c}{u_*\rho^{\frac{d+2}{2}}}\left(\frac{R\left(\Theta_{d}\right)^{\frac{d-2}{2}}}{\left|\Theta_{1}\right|^{\frac{d-2}{2}}}\cos\left(\rho h\left(\beta_*\right)+\frac{d-2}{4}\pi\right)+O\left(\frac{1}{\rho\left(|\Theta_{d}|-\beta\right)^{2}}\right)+O\left(\varepsilon\right)\right)
\end{align*}
where $c$ is a suitable constant and $u_*$ is as in \eqref{eq:u ast}.
\end{prop}

\begin{proof}
By Lemma \ref{lem:Fresnel} and since $u_*\approx |\Theta_d|-\beta$, we have
\[
\mathfrak{F}\left(\rho^{1/2}u_*\right)=-e^{i\rho u_*^{2}}\frac{1}{2i\rho^{1/2}u_*}+O\left(\frac{1}{\rho^{3/2}\left(|\Theta_{d}|-\beta\right)^{3}}\right).
\]
Thus, by Proposition \ref{prop:Stima beta beta+eps} we obtain
\begin{align*}
\widehat{\chi_{\mathcal{B}}}\left(\rho\Theta\right)&= c\frac{R\left(\Theta_{d}\right)^{\frac{d-2}{2}}}{\rho^{\frac{d+1}{2}}\left|\Theta_{1}\right|^{\frac{d-2}{2}}}\mathrm{Re}\left(e^{id\frac{\pi}{4}}e^{i\rho h\left(\Theta_{d}\right)}e^{i\rho u_*^{2}}\frac{1}{2i\rho^{1/2}u_*}\right)\\
&\quad +\frac{1}{\rho^{\frac{d+4}{2}}}O\left(\frac{1}{\left(|\Theta_{d}|-\beta\right)^{3}}\right)+O\left(\frac{1}{\rho^{\frac{d+2}{2}}}\right)\\
&= \frac{c}{u_*\rho^{\frac{d+2}{2}}}\left(\frac{R\left(\Theta_{d}\right)^{\frac{d-2}{2}}}{\left|\Theta_{1}\right|^{\frac{d-2}{2}}}\cos\left(\rho h\left(\beta_*\right)+\frac{d-2}{4}\pi\right)+O\left(\frac{1}{\rho\left(|\Theta_{d}|-\beta\right)^{2}}\right)+O\left(u_*\right)\right). 
\end{align*}
Since $u_*\leqslant c\varepsilon$ the proposition follows. 
\end{proof}

\subsection{The region $\beta+\varepsilon<\left|\Theta_{d}\right|\leqslant1$.}

We start with a lemma.
\begin{lem}
\label{lem:3}Let
\begin{align*}
I\left(\rho,\Theta_{1}\right)&= \int_{-\beta}^{\beta}e^{-2\pi i\rho\Theta_{d}x}J_{\frac{d-1}{2}}\left(2\pi\rho\left|\Theta_{1}\right|R\left(x\right)\right)R\left(x\right)^{\frac{d-1}{2}}dx.
\end{align*}
Then
\[
I\left(\rho,\Theta_{1}\right)=R\left(\beta\right)^{\frac{d-1}{2}}J_{\frac{d-1}{2}}\left(2\pi\rho\left|\Theta_{1}\right|R\left(\beta\right)\right)\frac{\sin\left(2\pi\rho\Theta_{d}\beta\right)}{\pi\rho\Theta_{d}}+O\left(\frac{\left(\rho\left|\Theta_{1}\right|\right)^{\frac{d-1}{2}}}{\left(\rho\Theta_{d}\right)^{2}}\right).
\]
\end{lem}

\begin{proof}
Since for every $\nu$ we have $\frac{d}{dz}\left(J_{\nu}\left(z\right)z^{\nu}\right)=J_{\nu-1}\left(z\right)z^{\nu}$, and 
\begin{align}
|J_{\nu}\left(z\right)|\leqslant c_{\nu}|z|^{\nu},\label{eq:stimaBessel}  
\end{align}
see \cite[10.14.4]{DLMF}, integrating by parts yields
\begin{align*}
I\left(\rho,\Theta_{1}\right)=& R\left(\beta\right)^{\frac{d-1}{2}}J_{\frac{d-1}{2}}\left(2\pi\rho\left|\Theta_{1}\right|R\left(\beta\right)\right)\frac{\sin\left(2\pi\rho\Theta_{d}\beta\right)}{\pi\rho\Theta_{d}}\\
 & -\int_{-\beta}^{\beta}\frac{e^{-2\pi i\rho\Theta_{d}x}}{-2\pi i\rho\Theta_{d}}J_{\frac{d-3}{2}}\left(2\pi\rho\left|\Theta_{1}\right|R\left(x\right)\right)\left[R\left(x\right)\right]^{\frac{d-1}{2}}2\pi\rho\left|\Theta_{1}\right|R'\left(x\right)dx\\
=& R\left(\beta\right)^{\frac{d-1}{2}}J_{\frac{d-1}{2}}\left(2\pi\rho\left|\Theta_{1}\right|R\left(\beta\right)\right)\frac{\sin\left(2\pi\rho\Theta_{d}\beta\right)}{\pi\rho\Theta_{d}}-I_{1}\left(\rho,\Theta_{1}\right).
\end{align*}
To estimate $I_{1}\left(\rho,\Theta_{1}\right)$ we integrate by parts
one more time. We have
\begin{align*}
I_{1}\left(\rho,\Theta_{1}\right)=& i\frac{\rho\left|\Theta_{1}\right|}{\rho\Theta_{d}}\int_{-\beta}^{\beta}e^{-2\pi i\rho\Theta_{d}x}J_{\frac{d-3}{2}}\left(2\pi\rho\left|\Theta_{1}\right|R\left(x\right)\right)\left[R\left(x\right)\right]^{\frac{d-3}{2}}R\left(x\right)R'\left(x\right)dx\\
=& i\frac{\rho\left|\Theta_{1}\right|}{\rho\Theta_{d}}\left[\frac{e^{-2\pi i\rho\Theta_{d}x}}{-2\pi i\rho\Theta_{d}}J_{\frac{d-3}{2}}\left(2\pi\rho\left|\Theta_{1}\right|R\left(x\right)\right)\left[R\left(x\right)\right]^{\frac{d-3}{2}}R\left(x\right)R'\left(x\right)\right]_{-\beta}^{\beta}\\
 & -i\frac{\rho\left|\Theta_{1}\right|}{\rho\Theta_{d}}\int_{-\beta}^{\beta}\frac{e^{-2\pi i\rho\Theta_{d}x}}{-2\pi i\rho\Theta_{d}}\frac{d}{dx}\left[J_{\frac{d-3}{2}}\left(2\pi\rho\left|\Theta_{1}\right|R\left(x\right)\right)\left[R\left(x\right)\right]^{\frac{d-3}{2}}R\left(x\right)R'\left(x\right)\right]dx\\
=& I_{1,1}\left(\rho,\Theta_{1}\right)-I_{1,2}\left(\rho,\Theta_{1}\right).
\end{align*}
For the first term, by \eqref{eq:stimaBessel}, we have 
\[
\left|I_{1,1}\left(\rho,\Theta_{1}\right)\right|\leqslant c\frac{\left(\rho\left|\Theta_{1}\right|\right)^{\frac{d-1}{2}}}{\rho^{2}\left|\Theta_{d}\right|^{2}}.
\]
Since
\begin{align*}
 & \frac{d}{dx}\left[J_{\frac{d-3}{2}}\left(2\pi\rho\left|\Theta_{1}\right|R\left(x\right)\right)\left[R\left(x\right)\right]^{\frac{d-3}{2}}R\left(x\right)R'\left(x\right)\right]\\
= & \,J_{\frac{d-5}{2}}\left(2\pi\rho\left|\Theta_{1}\right|R\left(x\right)\right) R\left(x\right)^{\frac{d-3}{2}}2\pi\rho\left|\Theta_{1}\right|R\left(x\right)R'\left(x\right)^{2}\\
 &  +J_{\frac{d-3}{2}}\left(2\pi\rho\left|\Theta_{1}\right|R\left(x\right)\right)R\left(x\right)^{\frac{d-3}{2}}\frac{d}{dx}\left[R\left(x\right)R'\left(x\right)\right],
\end{align*}
for the second term, again using \eqref{eq:stimaBessel}, we have
\begin{align*}
\left|I_{1,2}\left(\rho,\Theta_{1}\right)\right|\leqslant & \,c\frac{\rho\left|\Theta_{1}\right|}{\left(\rho\Theta_{d}\right)^{2}}\int_{-\beta}^{\beta}\left|J_{\frac{d-5}{2}}\left(2\pi\rho\left|\Theta_{1}\right|R\left(x\right)\right)R\left(x\right)^{\frac{d-3}{2}}\rho\left|\Theta_{1}\right|R\left(x\right)R'\left(x\right)^2\right|dx\\
 & +c\frac{\rho\left|\Theta_{1}\right|}{\left(\rho\Theta_{d}\right)^{2}}\int_{-\beta}^{\beta}\left|J_{\frac{d-3}{2}}\left(2\pi\rho\left|\Theta_{1}\right|R\left(x\right)\right)R\left(x\right)^{\frac{d-3}{2}}\left(R'\left(x\right)^{2}+R\left(x\right)R''\left(x\right)\right)\right|dx\\
\leqslant & \,c\frac{\left(\rho\left|\Theta_{1}\right|\right)^{\frac{d-1}{2}}}{\left(\rho\Theta_{d}\right)^{2}}\int_{-\beta}^{\beta}\left|R\left(x\right)^{d-3}R'\left(x\right)^{2}\right|dx\\
 & +c\frac{\left(\rho\left|\Theta_{1}\right|\right)^{\frac{d-1}{2}}}{\left(\rho\Theta_{d}\right)^{2}}\int_{-\beta}^{\beta}\left|R\left(x\right)^{d-3}\left(R'\left(x\right)^{2}+R\left(x\right)R''\left(x\right)\right)\right|dx\\
=& \,O\left(\frac{\left(\rho\left|\Theta_{1}\right|\right)^{\frac{d-1}{2}}}{\left(\rho\Theta_{d}\right)^{2}}\right).
\end{align*}
The lemma follows.
\end{proof}
\begin{prop}
\label{lem:4}Let $\beta+\varepsilon<\left|\Theta_{d}\right|\leqslant1$.
Then
\begin{align*}
\widehat{\chi_{\mathcal{B}}}\left(\rho\Theta\right)&= R\left(\beta\right)^{d-1}\frac{J_{\frac{d-1}{2}}\left(2\pi\rho\left|\Theta_{1}\right|R\left(\beta\right)\right)}{\left|\rho\Theta_{1}\right|^{\frac{d-1}{2}}R\left(\beta\right)^{\frac{d-1}{2}}}\frac{\sin\left(2\pi\rho\Theta_{d}\beta\right)}{\pi\rho\Theta_{d}}+O\left(\frac{1}{\rho^{2}}\right).
\end{align*}
In particular there exist $K>0$ and $c>0$ such that for $\left|2\pi\rho\Theta_{1} R\left(\beta\right)\right|<K$
we have
\begin{equation}\label{eq:small}
\left|\widehat{\chi_{\mathcal{B}}}\left(\rho\Theta\right)\right|\geqslant cR\left(\beta\right)^{d-1}\frac{\left|\sin\left(2\pi\rho\Theta_{d}\beta\right)\right|}{\left|\pi\rho\Theta_{d}\right|}+O\left(\frac{1}{\rho^{2}}\right).
\end{equation}
\end{prop}

\begin{proof}
Using Lemma \ref{lem:Fourier int Bessel}  and Lemma \ref{lem:3}, we have
\begin{align*}
\widehat{\chi_{\mathcal{B}}}\left(\rho\Theta\right)&= \frac{R\left(\beta\right)^{\frac{d-1}{2}}J_{\frac{d-1}{2}}\left(2\pi\rho\left|\Theta_{1}\right|R\left(\beta\right)\right)}{\left|\rho\Theta_{1}\right|^{\frac{d-1}{2}}}\frac{\sin\left(2\pi\rho\Theta_{d}\beta\right)}{\pi\rho\Theta_{d}}+O\left(\frac{1}{\rho^{2}}\right)\\
&= \left(2\pi\right)^{\frac{d-1}{2}}R\left(\beta\right)^{d-1}\frac{J_{\frac{d-1}{2}}\left(2\pi\rho\left|\Theta_{1}\right|R\left(\beta\right)\right)}{\left|2\pi\rho\Theta_{1}R\left(\beta\right)\right|^{\frac{d-1}{2}}}\frac{\sin\left(2\pi\rho\Theta_{d}\beta\right)}{\pi\rho\Theta_{d}}+O\left(\frac{1}{\rho^{2}}\right).
\end{align*}
From the power expansion of $J_{\frac{d-1}{2}}$, see e.g. \cite[10.2.2]{DLMF},
it follows that there exist $K$ and $c_1$ such that if $\left|2\pi\rho\Theta_{1}R\left(\beta\right)\right|<K$,
then 
\[
\frac{J_{\frac{d-1}{2}}\left(2\pi\rho\left|\Theta_{1}\right|R\left(\beta\right)\right)}{\left|2\pi\rho\Theta_{1}R\left(\beta\right)\right|^{\frac{d-1}{2}}}\geqslant c_{1}.
\]
Hence \eqref{eq:small} follows.
\end{proof}
We now consider the case $\rho\left|\Theta_{1}\right|$ large. 
\begin{prop}
\label{prop:Due cos}Let $\beta+\varepsilon<\left|\Theta_{d}\right|\leqslant1$
and $\rho\left|\Theta_{1}\right|$ large. There exists a constant
$c$ such that 
\begin{align}
\widehat{\chi_{\mathcal{B}}}\left(\rho\Theta\right) & =c\frac{R\left(\beta\right)^{\frac{d-2}{2}}}{\rho\left|\rho\Theta_{1}\right|^{\frac{d}{2}}}\left(\frac{1}{\Theta_{d}+\left|\Theta_{1}\right|\frac{\beta}{\sqrt{1-\beta^{2}}}}\cos\left(\rho h(-\beta)+\left(d-2\right)\frac{\pi}{4}\right)\right.\nonumber \\
 & \left.-\frac{1}{\Theta_{d}-\left|\Theta_{1}\right|\frac{\beta}{\sqrt{1-\beta^{2}}}}\cos\left(\rho h\left(\beta\right)+\left(d-2\right)\frac{\pi}{4}\right)\right)\label{eq:due cos} +O_\varepsilon\left(\frac{1}{\rho\left|\rho\Theta_{1}\right|^{\frac{d+2}{2}}}\right). 
\end{align}
\end{prop}

\begin{proof}
By Lemma
\ref{lem:Fourier int Bessel} and Lemma \ref{lem:5} we have 
\begin{align}
\widehat{\chi_{\mathcal{B}}}\left(\rho\Theta\right) =&\frac{c_{1}}{\left|\rho\Theta_{1}\right|^{\frac{d}{2}}}\mathrm{Re}\left(e^{id\frac{\pi}{4}}\int_{-\beta}^{\beta}e^{-2\pi i\rho\left(\Theta_{d}x+\left|\Theta_{1}\right|R\left(x\right)\right)}R\left(x\right)^{\frac{d-2}{2}}dx\right)\nonumber \\
 & +\frac{c_{2}}{\left|\rho\Theta_{1}\right|^{\frac{d+2}{2}}}\mathrm{Im}\left(e^{id\frac{\pi}{4}}\int_{-\beta}^{\beta}e^{-2\pi i\rho\left(\Theta_{d}x+\left|\Theta_{1}\right|R\left(x\right)\right)}R\left(x\right)^{\frac{d-4}{2}}dx\right)\label{eq:rhoThetaLarge}\\
 & +O\left(\frac{1}{\rho\left|\rho\Theta_{1}\right|^{\frac{d+2}{2}}}\right).\nonumber 
\end{align}
Let us consider the integral in the first term 
\begin{align*}
H\left(\rho\right) & =\int_{-\beta}^{\beta}e^{-2\pi i\rho\left(\Theta_{d}x+\left|\Theta_{1}\right|R\left(x\right)\right)}R\left(x\right)^{\frac{d-2}{2}}dx =\int_{-\beta}^{\beta}e^{i\rho h\left(x\right)}R\left(x\right)^{\frac{d-2}{2}}dx.
\end{align*}
Notice that since $\beta+\varepsilon<|\Theta_{d}|\leqslant1$, $h'(x)$ does not vanish if
$x\in[-\beta,\beta]$. More precisely
we have 
\begin{align*}
\inf_{\left[-\beta,\beta\right]}\left|h'\left(x\right)\right| = & \, 2\pi\inf_{\left[-\beta,\beta\right]}\left|\Theta_{d}+\left|\Theta_{1}\right|\frac{-x}{\sqrt{1-x^{2}}}\right|\geqslant2\pi\left(|\Theta_{d}|-\sqrt{1-\Theta_{d}^{2}}\frac{\beta}{\sqrt{1-\beta^{2}}}\right)\\
\geqslant&  \, 2\pi\left(\beta+\varepsilon-\frac{\sqrt{1-\Theta_{d}^{2}}}{\sqrt{1-\beta^{2}}}\beta\right)\geqslant2\pi\varepsilon.
\end{align*}
Integrating by parts yields 
\begin{align*}
H\left(\rho\right) & =\left[e^{i\rho h\left(x\right)}\frac{R\left(x\right)^{\frac{d-2}{2}}}{\rho ih'\left(x\right)}\right]_{-\beta}^{\beta}-\int_{-\beta}^{\beta}e^{i\rho h\left(x\right)}\frac{d}{dx}\left(\frac{R\left(x\right)^{\frac{d-2}{2}}}{\rho ih'\left(x\right)}\right)dx\\
 & =e^{i\rho h\left(\beta\right)}\frac{R\left(\beta\right)^{\frac{d-2}{2}}}{\rho ih'\left(\beta\right)}-e^{i\rho h\left(-\beta\right)}\frac{R\left(-\beta\right)^{\frac{d-2}{2}}}{\rho ih'\left(-\beta\right)}-\frac{1}{\rho}\int_{-\beta}^{\beta}e^{i\rho h\left(x\right)}\frac{d}{dx}\left(\frac{R\left(x\right)^{\frac{d-2}{2}}}{ih'\left(x\right)}\right)dx\\
 & =\frac{R\left(\beta\right)^{\frac{d-2}{2}}}{i\rho}\left[\frac{e^{i\rho h\left(\beta\right)}}{h'\left(\beta\right)}-\frac{e^{i\rho h\left(-\beta\right)}}{h'\left(-\beta\right)}\right]-O_\varepsilon\left(\rho^{-2}\right).
\end{align*}
For the first term in (\ref{eq:rhoThetaLarge}) we have
\begin{align*}
 & \frac{c_{1}}{\left|\rho\Theta_{1}\right|^{\frac{d}{2}}}\mathrm{Re}\left(e^{id\frac{\pi}{4}}\int_{-\beta}^{\beta}e^{-2\pi i\rho\left(\Theta_{d}x+\left|\Theta_{1}\right|R\left(x\right)\right)}R\left(x\right)^{\frac{d-2}{2}}dx\right)\\
= & \frac{c_{1}R\left(\beta\right)^{\frac{d-2}{2}}}{2\pi\rho\left|\rho\Theta_{1}\right|^{\frac{d}{2}}}\mathrm{Re}\left(\frac{e^{i\rho h\left(-\beta\right)+i\left(d-2\right)\frac{\pi}{4}}}{\Theta_{d}+\left|\Theta_{1}\right|\frac{\beta}{\sqrt{1-\beta^{2}}}}-\frac{e^{i\rho h\left(\beta\right)+i\left(d-2\right)\frac{\pi}{4}}}{\Theta_{d}-\left|\Theta_{1}\right|\frac{\beta}{\sqrt{1-\beta^{2}}}}\right)+O_\varepsilon\left(\frac{1}{\rho^2\left|\rho\Theta_{1}\right|^{\frac{d}{2}}}\right)\\
= & \frac{c_{1}R\left(\beta\right)^{\frac{d-2}{2}}}{2\pi\rho\left|\rho\Theta_{1}\right|^{\frac{d}{2}}}\left(\frac{1}{\Theta_{d}+\left|\Theta_{1}\right|\frac{\beta}{\sqrt{1-\beta^{2}}}}\cos\left(\rho h\left(-\beta\right)+\left(d-2\right)\frac{\pi}{4}\right)\right.\\
 & \left.-\frac{1}{\Theta_{d}-\left|\Theta_{1}\right|\frac{\beta}{\sqrt{1-\beta^{2}}}}\cos\left(\rho h\left(\beta\right)+\left(d-2\right)\frac{\pi}{4}\right)\right)+O_\varepsilon\left(\frac{1}{\rho^2\left|\rho\Theta_{1}\right|^{\frac{d}{2}}}\right).
\end{align*}
A similar computation shows that for the second term we have 
\begin{align*}
 & \frac{c_{1}}{\left|\rho\Theta_{1}\right|^{\frac{d+2}{2}}}\mathrm{Im}\left(e^{id\frac{\pi}{4}}\int_{-\beta}^{\beta}e^{-2\pi i\rho\left(\Theta_{d}x+\left|\Theta_{1}\right|R\left(x\right)\right)}R\left(x\right)^{\frac{d-4}{2}}dx\right)=O_\varepsilon\left(\frac{1}{\rho\left|\rho\Theta_{1}\right|^{\frac{d+2}{2}}}\right).
\end{align*}
\end{proof}

\section{Proofs of Theorem \ref{thm:0} and Theorem \ref{thm:above}}\label{sect:proof}

We first prove the lower bounds stated in Theorem \ref{thm:0}; the argument is split into several propositions, where we shall repeatedly use the inequality
\[
|A+B|^2 \geqslant \frac12 |A|^2-|B|^2.
\]
The proof of Theorem \ref{thm:above} is then obtained from the pointwise asymptotics of Theorem \ref{thm:1}.
\begin{prop}
Let $\left|\Theta_{d}\right|\leqslant\beta$, then there exists a positive constant $c$ such that for $A$ sufficiently large we have 
\[
\frac{1}{A}\int_{A}^{2A}\left|\widehat{\chi_{\mathcal{B}}}\left(\rho\Theta\right)\right|^{2}d\rho\geqslant\frac{c}{A^{d+1}}.
\]
\end{prop}

\begin{proof}
By point (1) of Theorem \ref{thm:1} we have
\[
\widehat{\chi_{\mathcal{B}}}\left(\rho\Theta\right)=\frac{c}{\rho^{\frac{d+1}{2}}}\frac{R\left(\Theta_{d}\right)^{\frac{d-2}{2}}}{\left|\Theta'\right|^{\frac{d-2}{2}}}\left|\Xi\left(\rho,\Theta\right)\right|\cos\left(\rho h(\Theta_d)+\vartheta\left(\rho,\Theta\right)+d\frac{\pi}{4}\right)+O\left(\frac{1}{\rho^{\frac{d+2}{2}}}\right)
\]
with $\Xi\left(\rho,\Theta\right)$ and $\vartheta\left(\rho,\Theta\right)$ satisfying
\[
\left|\Xi\left(\rho,\Theta\right)\right|\approx1
\]
and
\[
\frac{\pi}{4}-\delta\leqslant\vartheta\left(\rho,\Theta\right)\leqslant\frac{\pi}{4}+\delta
\]
with $\delta<\frac{\pi}{4}$. Hence
\begin{align*}
 & \frac{1}{A}\int_{A}^{2A}\left|\widehat{\chi_{\mathcal{B}}}\left(\rho\Theta\right)\right|^{2}d\rho\\
\geqslant & \frac{c}{A^{d+2}}\int_{A}^{2A}\left|\frac{R\left(\Theta_{d}\right)^{\frac{d-2}{2}}}{\left|\Theta'\right|^{\frac{d-2}{2}}}\left|\Xi\left(\rho,\Theta\right)\right|\cos\left(\rho h(\Theta_d)+\vartheta\left(\rho,\Theta\right)+d\frac{\pi}{4}\right)\right|^{2}d\rho+O\left(\frac{1}{A^{d+2}}\right)\\
\geqslant & \frac{c}{A^{d+2}}\int_{A}^{2A}\left|\cos\left(\rho h(\Theta_d)+\vartheta\left(\rho,\Theta\right)+d\frac{\pi}{4}\right)\right|^{2}d\rho+O\left(\frac{1}{A^{d+2}}\right).
\end{align*}
For simplicity in the rest of the proof the dependence of $\vartheta$ on $\Theta$ is left implicit. For the above integral we have
\begin{align*}
 & \int_{A}^{2A}\left|\cos\left(2\pi\rho\left(\Theta_{d}^{2}+\left|\Theta'\right|R\left(\Theta_{d}\right)\right)-\vartheta\left(\rho\right)-d\frac{\pi}{4}\right)\right|^{2}d\rho\\
&= \frac{1}{2\pi\left(\Theta_{d}^{2}+\left|\Theta'\right|R\left(\Theta_{d}\right)\right)}\int_{U_{1}}^{U_{2}}\left|\cos\left(u-\vartheta\left(\rho\left(u\right)\right)\right)\right|^{2}du
\end{align*}
with $\rho\left(u\right)=\left[2\pi\left(\Theta_{d}^{2}+\left|\Theta'\right|R\left(\Theta_{d}\right)\right)\right]^{-1}\left(u+d\frac{\pi}{4}\right)$,
\[
U_{1}=2\pi A\left(\Theta_{d}^{2}+\left|\Theta'\right|R\left(\Theta_{d}\right)\right)-d\frac{\pi}{4},
\]
and 
\[
U_{2}=4\pi A\left(\Theta_{d}^{2}+\left|\Theta'\right|R\left(\Theta_{d}\right)\right)-d\frac{\pi}{4}.
\]
Letting $m_{1}=\left[U_{1}/\pi\right]+1$ and $m_{2}=\left[U_{2}/\pi\right]-1,$
we obtain
\begin{align*}
\int_{U_{1}}^{U_{2}}\left|\cos\left(u-\vartheta\left(\rho\left(u\right)\right)\right)\right|^{2}du\geqslant 
&\sum_{k=m_{1}}^{m_{2}-1}\int_{0}^{\pi}\left|\cos\left(u-\vartheta\left(\rho\left(u+k\pi\right)\right)\right)\right|^{2}du\\
\geqslant & \sum_{k=m_{1}}^{m_{2}-1}\int_{0}^{\frac{\pi}{2}}\left|\cos\left(u-\vartheta\left(\rho\left(u+k\pi\right)\right)\right)\right|^{2}du.
\end{align*}
Since in the above integral $u-\vartheta\left(\rho\left(u+k\pi\right)\right)\in\left(-\frac{\pi}{4}-\delta,\frac{\pi}{4}+\delta\right)$
we obtain
\[
\int_{U_{1}}^{U_{2}}\left|\cos\left(u-\vartheta\left(\rho\left(u\right)\right)\right)\right|^{2}du\geqslant c\left(m_{2}-m_{1}\right)\geqslant cA.
\]
It follows that for $A$ large enough
\begin{align*}
\frac{1}{A}\int_{A}^{2A}\left|\widehat{\chi_{\mathcal{B}}}\left(\rho\Theta\right)\right|^{2}d\rho\geqslant & \frac{c}{A^{d+1}}+O\left(\frac{1}{A^{d+2}}\right)\geqslant\frac{c}{A^{d+1}}.
\end{align*}
\end{proof}
\begin{prop}
Let $L>0$ and $0<\varepsilon<1-\beta$. Then there exists $c>0$ such that for $\beta<\left|\Theta_{d}\right|\leqslant\beta+\varepsilon$,  $A^{1/2}\left(\left|\Theta_{d}\right|-\beta\right)\leqslant L$ and $A$ large enough,  we have
\[
\frac{1}{A}\int_{A}^{2A}\left|\widehat{\chi_{\mathcal{B}}}\left(\rho\Theta\right)\right|^{2}d\rho\geqslant\frac{c}{A^{d+1}}.
\]
\end{prop}

\begin{proof}
By point (2) of Theorem \ref{thm:1} we have
\begin{align*}
\widehat{\chi_{\mathcal{B}}}\left(\rho\Theta\right) & =c\frac{R\left(\Theta_{d}\right)^{\frac{d-2}{2}}}{\rho^{\frac{d+1}{2}}\left|\Theta'\right|^\frac{d-2}{2}}\mathrm{Re}\left(e^{id\frac{\pi}{4}}e^{i\rho h\left(\Theta_{d}\right)}\mathfrak{F}\left(\rho^{1/2}u_{*}\right)\right)+O\left(\frac{1}{\rho^{\frac{d+2}{2}}}\right).
\end{align*}
It follows that
\[
\frac{1}{A}\int_{A}^{2A}\left|\widehat{\chi_{\mathcal{B}}}\left(\rho\Theta\right)\right|^{2}d\rho\geqslant c\frac{R\left(\Theta_{d}\right)^{d-2}}{\left|\Theta'\right|^{d-2}A^{d+2}}\int_{A}^{2A}\left|\mathrm{Re}\left(e^{id\frac{\pi}{4}}e^{i\rho h\left(\Theta_{d}\right)}\mathfrak{F}\left(\rho^{1/2}u_{*}\right)\right)\right|^{2}d\rho+O\left(\frac{1}{A^{d+2}}\right).
\]
Let 
\[
I_{k}=\left(\frac{k\pi+d\frac{\pi}{4}}{|h\left(\beta_*\right)|},\frac{k\pi+d\frac{\pi}{4}+\frac{\pi}{2}}{|h\left(\beta_*\right)|}\right).
\]
We claim that there exists $c_{1}>0$ such that if $\rho^{1/2}u_{*}\leqslant L$
then for every $k\in\mathbb{N}$ and $\rho\in I_{k}$ we have
\begin{equation}
\left|\mathrm{Re}\left(e^{id\frac{\pi}{4}}e^{i\rho h\left(\Theta_{d}\right)}\mathfrak{F}\left(\rho^{1/2}u_{*}\right)\right)\right|>c_{1}.\label{eq:Re da sotto-1}
\end{equation}
Set 
\[
J=\bigcup_k I_{k}\cap[A,2A]
\]
and observe that for $A$ large enough $\left|J\right|\geqslant c_{2}A$
for a suitable constant $c_{2}>0$. Then
\begin{align*}
\frac{1}{A}\int_{A}^{2A}\left|\widehat{\chi_{\mathcal{B}}}\left(\rho\Theta\right)\right|^{2}d\rho & \geqslant c\frac{1}{A^{d+2}}\int_{J}c_{1}^{2}d\rho+O\left(\frac{1}{A^{d+2}}\right)\geqslant c\frac{1}{A^{d+1}}.
\end{align*}
It remains to show that there exists $c_{1}>0$ such that (\ref{eq:Re da sotto-1})
holds when $\rho^{1/2}u_{1}\leqslant L$ and $\rho\in I_{k}$. We
will use the following representation of the Fresnel function $\mathfrak{F}\left(x\right)$
contained in \cite[Lemma 3.1]{Lobo-villalobos}
\[
\mathfrak{F}\left(x\right)=\frac{1}{2}e^{ix^{2}}\left[g\left(x^{2}\right)+if\left(x^{2}\right)\right],
\]
where
\[
g\left(t\right)=\frac{1}{\sqrt{\pi}}\int_{0}^{+\infty}\frac{e^{-ty}y^{1/2}}{1+y^{2}}dy
,\]
\[
f\left(t\right)=\frac{1}{\sqrt{\pi}}\int_{0}^{+\infty}\frac{e^{-ty}y^{-1/2}}{1+y^{2}}dy.
\]
It follows that
\begin{align*}
\mathrm{Re}\left(e^{id\frac{\pi}{4}}e^{i\rho h\left(\Theta_{d}\right)}\mathfrak{F}\left(\rho^{1/2}u_{*}\right)\right)&= \mathrm{Re}\left(e^{id\frac{\pi}{4}}e^{i\rho h\left(\beta_*\right)}\left[g\left(\rho u_{*}^{2}\right)+if\left(\rho u_{*}^{2}\right)\right]\right)\\
&= \cos\left(\rho \left|h\left(\beta_*\right)\right|-d\frac{\pi}{4}\right)g\left(\rho u_{*}^{2}\right)+\sin\left(\rho \left|h\left(\beta_*\right)\right|-d\frac{\pi}{4}\right)f\left(\rho u_{*}^{2}\right).
\end{align*}
Since   
$\cos\left(\rho \left|h\left(\beta_*\right)\right|-d\frac{\pi}{4}\right)$
and $\sin\left(\rho\left|h\left(\beta_*\right)\right|-d\frac{\pi}{4}\right)$ 
have the same sign for $\rho\in I_k$, and $f$ and $g$ are positive, we have
\begin{align*}
 & \left|\cos\left(\rho\left| h\left(\beta_*\right)\right|-d\frac{\pi}{4}\right)g\left(\rho u_{*}^{2}\right)+\sin\left(\rho \left|h\left(\beta_*\right)\right|-d\frac{\pi}{4}\right)f\left(\rho u_{*}^{2}\right)\right|\geqslant\min\left(f\left(\rho u_{*}^{2}\right),g\left(\rho u_{*}^{2}\right)\right).
\end{align*}
 Finally, for $\rho u_{*}^{2}\leqslant L$ both $f\left(\rho u_{*}^{2}\right)$
and $g\left(\rho u_{*}^{2}\right)$ are bounded away from zero and therefore
\[
\left|\mathrm{Re}\left(e^{id\frac{\pi}{4}}e^{i\rho h\left(\Theta_{d}\right)}\mathfrak{F}\left(\rho^{1/2}u_{*}\right)\right)\right|>c_{1}.
\]
\end{proof}
\begin{prop}
There exist positive constants $L$, $\varepsilon$, and $c$ such
that for $\beta<\left|\Theta_{d}\right|\leqslant\beta+\varepsilon$,
$A^{1/2}\left(\left|\Theta_{d}\right|-\beta\right)>L$ and $A$
large enough we have
\[
\frac{1}{A}\int_{A}^{2A}\left|\widehat{\chi_{\mathcal{B}}}\left(\rho\Theta\right)\right|^{2}d\rho\geqslant\frac{c}{A^{d+2}\left(\left|\Theta_{d}\right|-\beta\right)^{2}}.
\]
\end{prop}

\begin{proof}
By point (3) of Theorem \ref{thm:1} we have
\[
\widehat{\chi_{\mathcal{B}}}\left(\rho\Theta\right)=\frac{c}{u_{*}\rho^{\frac{d+2}{2}}}\left(\frac{R\left(\Theta_{d}\right)^{\frac{d-2}{2}}}{\left|\Theta'\right|^{\frac{d-2}{2}}}\cos\left(\rho h\left(\beta_*\right)+\frac{d-2}{4}\pi\right)+O\left(\frac{1}{\rho\left(|\Theta_{d}|-\beta\right)^{2}}\right)+O\left(\varepsilon\right)\right).
\]
It follows that
\begin{align*}
 & \frac{1}{A}\int_{A}^{2A}\left|\widehat{\chi_{\mathcal{B}}}\left(\rho\Theta\right)\right|^{2}d\rho\\
\geqslant & \frac{c}{A^{d+3}u_{*}^{2}}\left[\int_{A}^{2A}\left|\cos\left(\rho h\left(\beta_*\right)+\frac{d-2}{4}\pi\right)\right|^{2}d\rho+O\left(\frac{A}{A^{2}\left(|\Theta_{d}|-\beta\right)^{4}}\right)+O\left(A\varepsilon^{2}\right)\right].
\end{align*}
Since $\left|h\left(\beta_*\right) \right|$ is bounded away from zero, for the above integral we have 
\[
\int_{A}^{2A}\left|\cos\left(\rho h\left(\beta_*\right)+\frac{d-2}{4}\pi\right)\right|^{2}d\rho = \frac 12 A +O(1).
\]
Therefore
\begin{align*}
 & \frac{1}{A}\int_{A}^{2A}\left|\widehat{\chi_{\mathcal{B}}}\left(\rho\Theta\right)\right|^{2}d\rho\\
\geqslant & \frac{c}{A^{d+3}u_{*}^{2}}\left[\int_{A}^{2A}\left|\cos\left(\rho h\left(\beta_*\right)+\frac{d-2}{4}\pi\right)\right|^{2}d\rho+O\left(\frac{A}{A^{2}\left(|\Theta_{d}|-\beta\right)^{4}}\right)+O\left(A\varepsilon^{2}\right)\right]\\
\geqslant & \frac{c}{A^{d+3}u_{*}^{2}}\left[\frac{1}{2}A+O\left(\frac{A}{A^{2}\left(|\Theta_{d}|-\beta\right)^{4}}\right)+O\left(A\varepsilon^{2}\right)\right]\\
\geqslant & \frac{C}{A^{d+2}u_{*}^{2}}\left[1+O\left(\frac{1}{A^2\left(|\Theta_{d}|-\beta\right)^{4}}\right)+O\left(\varepsilon^{2}\right)\right]\geqslant\frac{C}{A^{d+2}u_{*}^{2}},
\end{align*}
if $\varepsilon$ is small enough and $A^2\left(|\Theta_{d}|-\beta\right)^{4}$
is large enough.
\end{proof}
\begin{prop}
Let $\beta+\varepsilon<\left|\Theta_{d}\right|\leqslant1$ then
\[
\frac{1}{A}\int_{A}^{2A}\left|\widehat{\chi_{\mathcal{B}}}\left(\rho\Theta\right)\right|^{2}d\rho\geqslant\frac{c}{A^{2}}\frac{1}{1+\left(A\left|\Theta'\right|\right)^{d}}.
\]
\end{prop}

\begin{proof}
Let $K$ as in point (4) of Theorem \ref{thm:1} and let $L$ be a large
constant. Assume first 
$4\pi AR\left(\beta\right)\left|\Theta'\right|<K$.
Then for $A$ large enough, by (\ref{eq:da sotto 4) thm 1}), we have
\begin{align*}
\frac{1}{A}\int_{A}^{2A}\left|\widehat{\chi_{\mathcal{B}}}\left(\rho\Theta\right)\right|^{2}d\rho\geqslant & \frac{c}{A}\int_{A}^{2A}\left(R\left(\beta\right)^{d-1}\right)^{2}\frac{\left|\sin\left(2\pi\rho\Theta_{d}\beta\right)\right|^{2}}{\left(\rho\Theta_{d}\right)^{2}}d\rho+O\left(\frac{1}{A^{4}}\right)\\
\geqslant & \frac{c}{A^{3}}\int_{A}^{2A}\left|\sin\left(2\pi\rho\Theta_{d}\beta\right)\right|^{2}d\rho+O\left(\frac{1}{A^{4}}\right)\geqslant \frac{c}{A^{2}}.
\end{align*}
Assume now $K\leqslant 4\pi AR\left(\beta\right) \left|\Theta'\right|< L$. 
Then, by point (4) in Theorem \ref{thm:1} we have
\begin{align*}
 & \frac{1}{A}\int_{A}^{2A}\left|\widehat{\chi_{\mathcal{B}}}\left(\rho\Theta\right)\right|^{2}d\rho\\
\geqslant & \frac{1}{A}\int_{A}^{2A}\left|R\left(\beta\right)^{d-1}\left(\frac{J_{\frac{d-1}{2}}\left(2\pi\rho\left|\Theta'\right|R\left(\beta\right)\right)}{\left|\rho\Theta'\right|^{\frac{d-1}{2}}R\left(\beta\right)^{\frac{d-1}{2}}}\right)\frac{\sin\left(2\pi\rho\Theta_{d}\beta\right)}{\pi\rho\Theta_{d}}\right|^{2}d\rho+O\left(\frac{1}{A^{4}}\right)\\
\geqslant & \frac{c}{A^{3}}\int_{A}^{2A}\left|J_{\frac{d-1}{2}}\left(2\pi\rho\left|\Theta'\right|R\left(\beta\right)\right)\sin\left(2\pi\rho\Theta_{d}\beta\right)\right|^{2}d\rho+O\left(\frac{1}{A^{4}}\right)\\
\geqslant & \frac{c}{LA^{3}}\int_{A}^{2A}2\pi\rho\left|\Theta'\right|R\left(\beta\right)\left|J_{\frac{d-1}{2}}\left(2\pi\rho\left|\Theta'\right|R\left(\beta\right)\right)\sin\left(2\pi\rho\Theta_{d}\beta\right)\right|^{2}d\rho+O\left(\frac{1}{A^{4}}\right).
\end{align*}
Let $\eta$ be a small positive number and let 
\[
H_{\eta}=\bigcup_{k\in\mathbb{Z}}\left(\frac{k-\eta}{2\left|\Theta_{d}\right|\beta},\frac{k+\eta}{2\left|\Theta_{d}\right|\beta}\right)
\]
and observe that if $\rho\notin H_{\eta}$ then $\left|\sin\left(2\pi\rho\Theta_{d}\beta\right)\right|>\sin\left(\pi\eta\right).$
Then
\begin{align*}
 & \int_{A}^{2A}2\pi\rho\left|\Theta'\right|R\left(\beta\right)\left|J_{\frac{d-1}{2}}\left(2\pi\rho\left|\Theta'\right|R\left(\beta\right)\right)\sin\left(2\pi\rho\Theta_{d}\beta\right)\right|^{2}d\rho\\
&\geqslant \sin^{2}\left(\pi\eta\right)\int_{\left[A,2A\right]\setminus H_{\eta}}2\pi\rho\left|\Theta'\right|R\left(\beta\right)\left|J_{\frac{d-1}{2}}\left(2\pi\rho\left|\Theta'\right|R\left(\beta\right)\right)\right|^{2}d\rho\\
&= \sin^{2}\left(\pi\eta\right)\left[\int_{A}^{2A}2\pi\rho\left|\Theta'\right|R\left(\beta\right)\left|J_{\frac{d-1}{2}}\left(2\pi\rho\left|\Theta'\right|R\left(\beta\right)\right)\right|^{2}d\rho\right.\\
&\quad\left.-\int_{\left[A,2A\right]\cap H_{\eta}}2\pi\rho\left|\Theta'\right|R\left(\beta\right)\left|J_{\frac{d-1}{2}}\left(2\pi\rho\left|\Theta'\right|R\left(\beta\right)\right)\right|^{2}d\rho\right].
\end{align*}
Notice that 
\begin{align*}
&\int_{A}^{2A}2\pi\rho\left|\Theta'\right|R\left(\beta\right)\left|J_{\frac{d-1}{2}}\left(2\pi\rho\left|\Theta'\right|R\left(\beta\right)\right)\right|^{2}d\rho
\\
=&\, \frac{A}{2\pi A\left|\Theta'\right|R\left(\beta\right)}\int_{2\pi A\left|\Theta'\right|R\left(\beta\right)}^{4\pi A\left|\Theta'\right|R\left(\beta\right)}u\left|J_{\frac{d-1}{2}}\left(u\right)\right|^{2}du\geqslant c_1 A,
\end{align*}
where 
\[
c_1=\min\limits_{z\in[K,L]}\frac{2}{z}\int_{z/2}^{z} u\left|J_{\frac{d-1}{2}}\left(u\right)\right|^{2}du>0.
\]
Hence, for a suitably small $\eta$, 
\begin{align*}
 & \frac{c}{LA^{3}}\int_{A}^{2A}2\pi\rho\left|\Theta'\right|R\left(\beta\right)\left|J_{\frac{d-1}{2}}\left(2\pi\rho\left|\Theta'\right|R\left(\beta\right)\right)\sin\left(2\pi\rho\Theta_{d}\beta\right)\right|^{2}d\rho\\
\geqslant & \frac{c\sin^{2}\left(\pi\eta\right)}{LA^{3}}\left(c_{1}A-\left|\left[A,2A\right]\cap H_{\eta}\right|\left\Vert u^{1/2}J_{\frac{d-1}{2}}\left(u\right)\right\Vert _{\infty}^{2}\right)\\
\geqslant & \frac{c\sin^{2}\left(\pi\eta\right)}{LA^{3}}\left(c_{1}A-c_{2}\eta A\left\Vert u^{1/2}J_{\frac{d-1}{2}}\left(u\right)\right\Vert _{\infty}^{2}\right)\geqslant \frac{c}{LA^{2}}.
\end{align*}
Finally we assume $4\pi AR\left(\beta\right)\left|\Theta'\right|>L$.
By (5) in Theorem \ref{thm:1} we have
\begin{align*}
\widehat{\chi_{\mathcal{B}}}\left(\rho\Theta\right) & =c\frac{R\left(\beta\right)^{\frac{d-2}{2}}}{\rho\left|\rho\Theta'\right|^{\frac{d}{2}}}\left[\frac{1}{\Theta_{d}+\left|\Theta'\right|\frac{\beta}{\sqrt{1-\beta^{2}}}}\cos\left(\rho h(-\beta)+\left(d-2\right)\frac{\pi}{4}\right)\right. \\
 & \left.-\frac{1}{\Theta_{d}-\left|\Theta'\right|\frac{\beta}{\sqrt{1-\beta^{2}}}}\cos\left(\rho h\left(\beta\right)+\left(d-2\right)\frac{\pi}{4}\right)\right] +O_\varepsilon\left(\frac{1}{\rho\left|\rho\Theta'\right|^{\frac{d+2}{2}}}\right). 
\end{align*}
It follows that
\begin{align*} & \int_{A}^{2A}\left|\widehat{\chi_{\mathcal{B}}}\left(\rho\Theta\right)\right|^{2}d\rho \geqslant \frac{cR\left(\beta\right)^{d-2}}{A^2\left|A\Theta'\right|^{d}} \left[\frac{1}{\left(\Theta_{d}+\left|\Theta'\right|\frac{\beta}{\sqrt{1-\beta^{2}}}\right)^{2}}\int_{A}^{2A}\cos^{2}\left(\rho h(-\beta)+\frac{d-2}{4}\pi\right)d\rho\right.\\
 & +\frac{1}{\left(\Theta_{d}-\left|\Theta'\right|\frac{\beta}{\sqrt{1-\beta^{2}}}\right)^{2}}\int_{A}^{2A}\cos^{2}\left(\rho h\left(\beta\right)+\frac{d-2}{4}\pi\right)d\rho\\
 & \left.-2\frac{1}{\left(\Theta_{d}^{2}-\left|\Theta'\right|^{2}\frac{\beta^{2}}{1-\beta^{2}}\right)}\int_{A}^{2A}\cos\left(\rho h(-\beta)+\frac{d-2}{4}\pi\right)\cos\left(\rho h\left(\beta\right)+\frac{d-2}{4}\pi\right)d\rho\vphantom{\frac{1}{\left(\Theta_{d}+\left|\Theta'\right|\frac{\beta}{\sqrt{1-\beta^{2}}}\right)^{2}}}\right]\\
 & +O\left(\frac{A}{A^2\left|A\Theta'\right|^{d+2}}\right).
 \end{align*}

The first two integrals can be computed explicitly and each gives a contribution of $\frac{1}{2}A+O\left(1\right)$. For
the third integral we have
\begin{align*}
 & \int_{A}^{2A}\cos\left(\rho h(-\beta)+\frac{d-2}{4}\pi\right)\cos\left(\rho h\left(\beta\right)+\frac{d-2}{4}\pi\right)d\rho\\
&= \frac{1}{2}\int_{A}^{2A}\cos\left(\rho\left(h(\beta)+h(-\beta)\right)+\frac{d-2}{2}\pi\right)d\rho+\frac{1}{2}\int_{A}^{2A}\cos\left(\rho\left(h(\beta)-h(-\beta)\right)\right)d\rho\\
&= \frac{1}{2}\left[\frac{1}{4\pi\left|\Theta'\right|R\left(\beta\right)}\sin\left(4\pi\rho\left|\Theta'\right|R\left(\beta\right)-\frac{d-2}{2}\pi\right)\right]_{A}^{2A}+\frac{1}{2}\left[\frac{1}{4\pi\Theta_{d}\beta}\sin\left(4\pi\rho\Theta_{d}\beta\right)\right]_{A}^{2A}\\
&= O\left(\frac{1}{\left|\Theta'\right|}\right)=O\left(\frac{A}{L}\right).
\end{align*}
Collecting all these estimates we obtain
\begin{align*}
 & \frac{1}{A}\int_{A}^{2A}\left|\widehat{\chi_{\mathcal{B}}}\left(\rho\Theta\right)\right|^{2}d\rho\\
\geqslant& \frac{c}{A}\frac{R\left(\beta\right)^{d-2}}{A^2\left|A\Theta'\right|^{d}}\left[\frac{\frac{1}{2}A}{\left(\Theta_{d}-\left|\Theta'\right|\frac{\beta}{\sqrt{1-\beta^{2}}}\right)^{2}}+\frac{\frac{1}{2}A}{\left(\Theta_{d}+\left|\Theta'\right|\frac{\beta}{\sqrt{1-\beta^{2}}}\right)^{2}}+O\left(\frac{1}{\Theta_{d}^{2}-\left|\Theta'\right|^{2}\frac{\beta^{2}}{1-\beta^{2}}}\frac{A}{L}\right)\right]\\
&+O\left(\frac{1}{A^2\left|A\Theta'\right|^{d+2}}\right)\\
=& \frac{cR\left(\beta\right)^{d-2}}{A^2\left|A\Theta'\right|^{d}}\left[\frac{\Theta_{d}^{2}+\left|\Theta'\right|^{2}\frac{\beta^{2}}{1-\beta^{2}}}{\left(\Theta_{d}^{2}-\left|\Theta'\right|^{2}\frac{\beta^{2}}{1-\beta^{2}}\right)^{2}}+O\left(\frac{1}{\Theta_{d}^{2}-\left|\Theta'\right|^{2}\frac{\beta^{2}}{1-\beta^{2}}}\frac{1}{L}\right)\right]+O\left(\frac{1}{A^2\left|A\Theta'\right|^{d+2}}\right)\\
=& \frac{cR\left(\beta\right)^{d-2}}{A^2\left|A\Theta'\right|^{d}}\frac{\Theta_{d}^{2}+\left|\Theta'\right|^{2}\frac{\beta^{2}}{1-\beta^{2}}}{\left(\Theta_{d}^{2}-\left|\Theta'\right|^{2}\frac{\beta^{2}}{1-\beta^{2}}\right)^{2}}\left[1+O\left(\frac{1}{L}\right)\right]+O\left(\frac{1}{A^2\left|A\Theta'\right|^{d+2}}\right)\geqslant \frac{c}{A^{2}\left|A\Theta'\right|^{d}}
\end{align*}
which holds for $L$ large enough.
\end{proof}

\begin{proof}[Proof of Theorem \ref{thm:above}]
The theorem follows directly from Theorem \ref{thm:1}.
Indeed, point (1) of the present theorem follows from points (1) and (2) of Theorem \ref{thm:1}, while point (2) follows from points (4) and (5) of Theorem \ref{thm:1}.
\end{proof}
\bigskip
\subsection*{Conflict of interest and data availability}
The authors state that there is no conflict of interest and that they do not analyse or generate any datasets. 

\bibliographystyle{abbrv}
\bibliography{biblio}
\end{document}